\documentclass[reqno,11pt]{amsart}

\usepackage[margin=3cm]{geometry}

\usepackage{amscd, amsbsy, amsfonts, amssymb, amsmath, amsthm, graphicx, fancyhdr, color, xy, enumitem, mathrsfs, mathtools, epsfig, setspace, ragged2e, microtype, mathptmx, caption, wasysym, pgfplots, tcolorbox, float, mdframed, framed, yhmath, url, hyperref}

\usepackage{tabularx}
\usepackage{booktabs}
\usepackage{float}
\usepackage{enumitem}

\usepgfplotslibrary{colormaps,fillbetween}
\pgfplotsset{
    colormap={slategraywhite}{
        rgb255=(112,128,144)
        rgb255=(255,159,101)
    }
}
\pgfplotsset{compat=1.15}

\usepackage{csquotes}

\newcommand{\be}{\begin{equation}}
\newcommand{\ee}{\end{equation}}
\newcommand{\C}{\mathbb{C}}
\newcommand{\N}{\mathbb{N}}
\newcommand{\Z}{\mathbb{Z}}
\newcommand{\R}{\mathbb{R}}
\newcommand{\cB}{\mathcal{B}}
\newcommand{\cT}{\mathcal{T}}
\newcommand{\cS}{\mathcal{S}}
\newcommand{\cD}{\mathcal{D}}
\newcommand{\zvec}{{\pmb{z}}}

\renewcommand{\Re}{\mathop{\mathrm{Re}}}
\renewcommand{\Im}{\mathop{\mathrm{Im}}}

\newcommand{\abs}[1]{\left\vert{#1}\right\vert}
\newcommand{\inn}[2]{\left\langle #1, #2 \right\rangle}
\newcommand{\prn}[1]{\left( #1 \right)}
\newcommand{\floor}[1]{\left\lfloor #1 \right\rfloor}
\newcommand{\ceil}[1]{\left\lceil #1 \right\rceil}
\newcommand{\wh}[1]{{\widehat{#1}}}
\newcommand{\p}{\partial}

\newcommand{\Span}{\mathrm{span\,}}

\newcommand{\Supp}{\mathrm{supp\,}}

\newtheorem{theorem}{Theorem}[section]
\theoremstyle{plain}

\newtheorem{definition}{Definition}[section]

\newtheorem{proposition}{Proposition}[section]
\newtheorem{remark}{Remark}[section]

\newtheorem*{notation*}{Notation}
\numberwithin{equation}{section}

\DeclareMathOperator\re{Re}
\DeclareMathOperator\im{Im}

\begin{document}

\title{Complex Box Splines}

\author{Fabio Marcelo Carvalho dos Santos}
\address{Department of Mathematics, Technical University of Munich, Boltzmannstr. 3, 85748 Garching b. Munich, Germany}
\email{ge37dey@mytum.de}

\author{Peter Massopust${}^1$}
\thanks{${}^1$Corresponding author}
\address{Department of Mathematics, Technical University of Munich, Boltzmannstr. 3, 85748 Garching b. Munich, Germany}
\email{massopust@ma.tum.de}

\begin{abstract}
The novel concept of box spline of complex degree is introduced and several of its properties derived and discussed. These box splines of complex degree generalize and extend the classical box splines. Relations to a class of fractional derivatives defined on certain function spaces are also exhibited.
\end{abstract}
%\vskip 12pt\noindent
%
\maketitle
\section{Introduction}\label{sec1}
The mathematical concept of splines has been a fertile ground for research and innovation, with complex B-splines and box splines being two prominent areas of study. Both have found applications in various domains, such as computer graphics, image and signal processing, and numerical analysis. However, the intersection of these two concepts, known as box splines of complex degree, for short, complex box splines presents an unexplored opportunity with potentially far-reaching implications.

The motivation behind this paper lies in the fusion of these two powerful mathematical constructs. By introducing complex box splines, this research aims to harness the strengths of both complex B-splines and box splines, creating a novel mathematical tool that could transform the way multivariate analysis is approached. While the potential applications of this concept are vast, the focus of this paper is strictly on the theoretical aspects and foundations of complex box splines. This includes the mathematical properties, construction techniques, and underlying theoretical principles. The aim is to create a comprehensive theoretical framework that bridges the gap between complex B-splines and box splines. Moreover, the existing literature, while rich in individual explorations of complex B-splines and box splines, lacks a unified theoretical study that connects these two areas. This paper seeks to fill this void, providing a systematic and purely theoretical approach to understanding complex box splines.

The first inspiration comes from the paper \cite{Xu2009} which introduces fractional box splines, the real counterparts to complex box splines, and an extension of the traditional concept. The aforementioned work explores the mathematical intricacies of fractional box splines, opening new perspectives. Amongst others, this approach will be extended to complex degrees.

The second key inspiration is drawn from the paper \cite{Forster2005} in which complex B-splines, more precisely cardinal polynomial B-splines of complex degree, were first introduced as a means of adding phase information to the approximants, thus creating a complex-valued transform for the purposes of signal and image analysis. (See \cite{F}, for a discussion of the effectiveness of such complex-valued transforms.) In \cite{Forster2010}, it was shown that complex B-splines are solutions of certain distributional fractional differential equations. A similar representation is derived in this paper for complex box splines extending the mathematical framework and forging a new path in the exploration of these mathematical constructs. 

The novel complex box splines allow the approximation of functions in $\R^d$, $d\in\N$, which have varying degrees of smoothness along different directions. Here, the reference \cite{Xu2009} comes to bear. However, at each point in $\R^d$, the complex box spline also assigns a direction on the unit sphere of $\R^{d}$ arising from the imaginary parts of the complex degrees. Data that require such directions can, for instance, be found in geophysics and are known as S (shear), L (Love) and R (Rayleigh) waves. Hence, complex box splines provide a higher-dimensional complex-valued transform.

The structure of this paper is as follows. In Section 2, complex truncated power functions and complex B-splines are reviewed and the properties discussed that are relevant for the purposes of this paper. The next section gives a brief summary of box splines and F-box splines using a distributional approach. Section 4 then defines multivariate truncated complex powers and complex box splines. Here, several relations between these two concepts are derived and their properties described. In Section 5, the relationship between complex box splines and a certain type of fractional derivative is exhibited demonstrating the flexibility of this new class of multivariate approximants.
\section{Complex Truncated Power Functions and Complex B-Splines}
In this section, we briefly review some of the properties of complex backward difference operators and complex B-splines. We also introduce complex one-dimensional truncated power functions whose extension to higher dimensions plays an important role later when complex box splines are defined. 

Throughout this paper, the set of natural numbers is denoted by $\N :=\{0,1, 2, \ldots\}$. We also set $\N_n := \{0,1, \ldots, n\}$.
%In the realm of mathematical analysis and signal processing, B-splines have proven to be invaluable tools. Cardinal B-splines, introduced by Schoenberg in \cite{Schoen1973}, have been particularly useful because of their simple form and compact support. However, the classical B-splines, being piecewise polynomial functions, have a fixed order of smoothness and approximation, which limits their flexibility.
%
%To overcome this limitation, researchers have introduced the concept of complex B-splines, where the degree of the B-spline is extended to the complex domain. This extension is not merely a theoretical exercise but has practical implications in various fields, including signal and image processing, computer graphics, and more.
%
%While there are multiple ways to define complex B-splines, such as through the Fourier domain, we will remain consistent with the previous chapter and introduce a \textit{complex backward difference operator}. The resulting complex B-splines, much like their classical counterparts, retain most of the important properties, such as smoothness and recurrence. However, they offer a greater degree of flexibility by allowing the order of smoothness and approximation to be adjusted continuously.
%
%In this chapter, informed by the pioneering work in \cite{Forster2005}, \cite{Unser2000}, and the subsequent advances in \cite{Forster2010}, we present an introduction to complex B-splines.

\subsection{Complex Backward Difference Operator}
In contrast to the approach taken when defining traditional B-splines, we introduce the complex backward difference operator directly in the setting of tempered distributions, rather than first defining it in the time domain. The reason for this choice is that it provides a more direct and streamlined exposition in the context of complex B-splines. This approach also allows us to immediately utilize the powerful and flexible framework provided by the theory of tempered distributions. Furthermore, this approach aligns well with the goal of this paper, which emphasizes the use of distribution theory in studying box splines. The references for this section are \cite{Forster2005,FM,Forster2010}.
%As a first step into the realm of complex B-splines, we introduce the complex backward difference operator. This operator serves as a key tool, acting as the linchpin that bridges real number operations into the complex domain and setting the stage for the subsequent introduction of complex B-splines. Let us inquire its definition and properties, paving the way for our exploration of complex B-splines.

\begin{definition}[Complex Backward Difference Operator]\label{ComplexBackward}
Let $T$ be a tempered distribution on $\R$ and $z$ be a complex number with $\Re(z) > -1$. The complex backward difference operator $\nabla^{z+1}$ of degree $z+1$ is defined as
\[
    \inn{\nabla^{z+1}T}{\varphi} \coloneqq \sum_{k\geq 0} (-1)^k \binom{z+1}{k} \inn{\tau_k T}{\varphi} = \sum_{k\geq 0} (-1)^k \binom{z+1}{k} \inn{T}{\tau_{-k}\varphi},
\]
where $\tau_{-k}\varphi:=\varphi(\cdot+k)$ and $\varphi\in\cS(\R)$.
\end{definition}

%\begin{remark}
%In contrast to the approach taken when defining B-splines, we introduce the complex backward difference operator directly in the setting of tempered distributions, rather than first defining it in the time domain. The reason for this choice is that it provides a more direct and streamlined exposition in the context of complex B-splines. This approach also allows us to immediately utilize the powerful and flexible framework provided by the theory of tempered distributions. Furthermore, this approach aligns well with the goal of this thesis, which emphasizes the use of distribution theory in studying B-splines.
%\end{remark}

\begin{proposition}\label{WellDefined}
The complex backward difference operator is well-defined for tempered distributions.
\end{proposition}
\begin{proof}
%We first recall that Schwartz functions are a class of rapidly decreasing functions that are very well behaved: not only are they smooth, but all their derivatives also decrease rapidly to zero at infinity, meaning that the terms $\tau_{-k}\varphi$ in the sum defining the complex backward difference operator become smaller and smaller as $k$ increases. 

%On the other hand, 

We need to ensure that the series defining the complex backward difference operator is absolutely convergent in order to apply linearity. Notice that 
\[
    \sum_{k\geq 0}\abs{\binom{z+1}{k}} < \infty.
\]
Indeed, for large $\abs{z}$, we have the following asymptotic behavior of the Gamma function:
\[
    \frac{\Gamma(z+a)}{\Gamma(z+b)}=z^{a-b}+O(\abs{z}^{-1}),
\]
for any curve that joins $z=0$ and $z=\infty$. 
%This asymptotic behavior of the gamma function has been studied in \cite{Asymptotic2012} and the reader is invited to consult this reference for further details. 
Thus,
\[
    \binom{z+1}{k} = \frac{\Gamma(z+2)}{\Gamma(k+1)\Gamma(z+2-k)} = \frac{1}{\Gamma(k+1)}z^k(1+O(\abs{z}^{-1}),
\]
wich implies
\[
    \sum_{k\geq 0}\abs{\binom{z+1}{k}}\sum_{k\geq 0}\abs{\frac{1}{\Gamma(k+1)}z^k(1+O(\abs{z}^{-1})} \leq Ce^{\abs{z}^{-1}},
\]
where the last inequality follows from the Taylor expansion of the exponential function and $C > 0$ is some constant.
\end{proof}

%Given the definition of the complex difference operator, we can now explore its properties. In particular, we can compute its Fourier transform, which provides valuable insights into its behavior.

The next proposition derives an expression for the Fourier transform of the complex difference operator which provides insights into its behavior. (See also, \cite{Unser2000} for a derivation for real $z$ and \cite{Forster2005} where the extension to complex $z$ is stated.)
We use the following form of the Fourier transform of $f\in L^1(\mathbb{R})$,
\[
    \widehat{f}(\omega) :=\int_{\mathbb{R}^n}f(x)e^{-i\omega x}\, dx,
\]
and extend it in the usual way to tempered distributions.
\begin{proposition}[Fourier Transform of Complex Difference Operators]
%\footnote{\cite{Unser2000}. Although this reference states the result for a real number $\alpha$, it can be extended to complex numbers $z$ as stated in \cite{Forster2005} and the properties of the complex difference operator.}]\label{FourierComplexBackward}
Let $T$ be a tempered distribution on $\R$. Then
\be\label{Fourierbackward}
    \inn{\wh{\nabla^{z+1} T}}{\varphi (\omega)} = \sum_{k\geq 0} (-1)^k \binom{z+1}{k} \inn{e^{-ik\omega}\wh{T}}{\varphi(\omega)},
\ee  
where $\varphi\in\cS(\R)$.
\end{proposition}
\noindent
\textbf{Note:} On occasion and in abuse of notation, we write the argument of a test function into the pairing if it clarifies the meaning of an expression.
\begin{proof}
As $\sum\limits_{k\geq 0}\abs{\binom{z+1}{k}}<\infty$ and the Fourier transform is linear, we have that
%\[
%    \inn{\wh{\nabla^{z+1} T}}{\varphi(\omega)} = \sum_{k\geq 0} (-1)^k \binom{z+1}{k} \inn{\tau_k\wh{T}}{\varphi}.
%\]
%Now, it remains to apply property (i) from \textbf{Proposition \ref{PropertiesFourier}} to conclude
\begin{align*}
    \inn{\wh{\nabla^{z+1} T}}{\varphi(\omega)} &=\sum_{k\geq 0} (-1)^k \binom{z+1}{k} \inn{\tau_k\wh{T}}{\varphi(\omega)} \\
    &=\sum_{k\geq 0} (-1)^k \binom{z+1}{k} \inn{e^{-ik\omega}\wh{T}}{\varphi(\omega)}.
\end{align*}
\end{proof}

We summarize some properties of complex difference operators and refer to \cite{Forster2005} for their proofs.

%Having established the Fourier transform of the complex difference operator, we can now explore some of its important properties. These properties will further illuminate the behavior of the operator and its interactions with other mathematical entities. The following proposition presents two such properties.
\begin{proposition}[Properties of Complex Difference Operators]
%\footnote{\cite{Forster2005}. Again, the author adapted the statements for tempered distributions.}]\label{PropertiesComplexBackward}
The complex backward difference operator satisfies the following properties:
\begin{enumerate}
    \item[(i)] If $T$ is a tempered distribution on $\R$, then
    \be\label{ComplexBackwardtT}
        \inn{\nabla^{z+1}(tT)}{\varphi} = \inn{(z+1)\nabla^z T}{\varphi}+\inn{[t-(z+1)]\nabla^{z+1}T}{\varphi}.
    \ee
    \item[(ii)] If $T_1$ and $T_2$ are tempered distributions on $\R$, then
    \be\label{ComplexBackwardConvolution}
        \inn{\nabla^{(z_1+1)+(z_2+1)}(T_1\ast T_2)}{\varphi}=\inn{\nabla^{z_1+1}T_1 \ast \nabla^{z_2+1}T_2}{\varphi}.
    \ee
\end{enumerate}
\end{proposition}

\subsection{Truncated Complex Power Functions}
In this section, we will define {truncated complex power functions} and present those properties which relate to complex B-splines and later to complex box splines.
% explore its properties, and understand its role in the definition of complex B-splines.

We extend the truncated power function $t\mapsto t_+^n$ to a complex degree $z$ by setting
\[
    t_+^z\coloneqq \begin{cases}
        t^z = e^{z\log(t)} = t^{\Re(z)}e^{i\Im(z)\log(t)}, & \text{if $t>0$}; \\
        0, & \text{else}.
    \end{cases}
\]

\begin{figure}[H]
    \centering
    \includegraphics[scale=0.4]{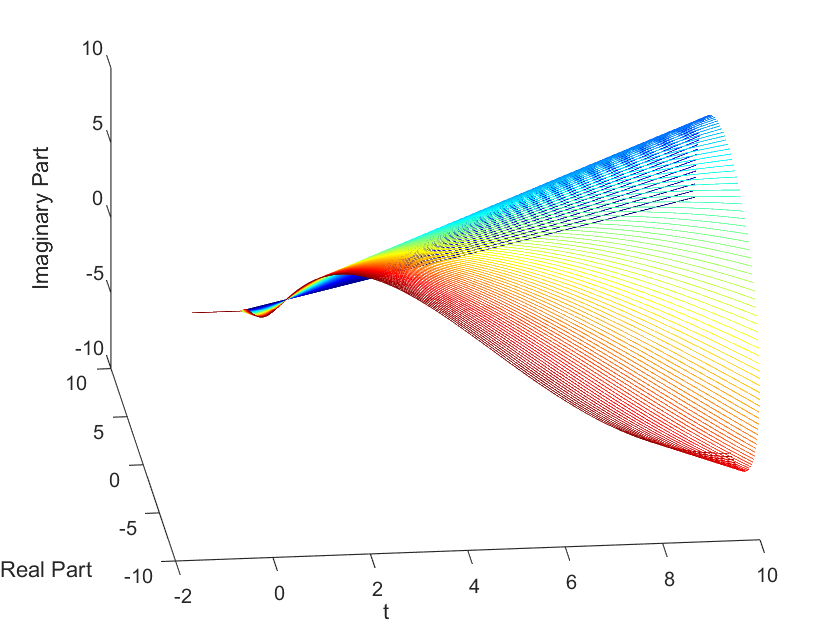}
    \caption{Truncated complex power function with $z=1+i\gamma$ for $\gamma\in[0,2]$.}
    \label{TruncatedComplexPower}
\end{figure}

The next result is a straightforward consequence of the definition of the complex truncated power function. The easy proof is left to the reader. 

\begin{theorem}\label{TruncatedComplexCn}
For $\Re(z)\geq 0$, we have
\[
    t_+^z\in C^{\floor{\Re(z)}}(\R),
\]
where $\floor{\cdot}$ denotes the floor function.
\end{theorem}

$t_+^z$ induces a tempered distribution on $\R$ by setting
\be\label{TruncatedComplexTempered}
    \inn{t_+^z}{\varphi(t)} := \int_0^\infty t^z \varphi(t) dt,
\ee
where $\varphi\in\cS(\R)$. As $t_+^z$ is locally integrable for $\Re(z)>-1$, the integral is well-defined for $\Re(z)>-1$. 

%From \textbf{Proposition \ref{TruncatedComplexCn}}, we know that we can differentiate $t_+^z$ $\floor{\Re(z)}$ times as long as $\Re(z)>0$ and its derivative is given by
%\[
%    \frac{d}{dt}t_+^z = zt_+^{z-1}.
%\]
%Using the distributional representation \eqref{TruncatedComplexTempered}, we can differentiate even further in the weak sense. 

\begin{proposition}[Weak Derivative of Truncated Complex Powers \cite{Gelfand1966}]
For $-1<\Re(z)<0$, we have
\be\label{ComplexTruncatedWeakDerivative}
    \inn{Dt_+^z}{\varphi(t)} = \inn{zt_+^{z-1}}{\varphi(t)-\varphi(0)}
\ee
where $D$ denotes the ordinary derivative.
\end{proposition}
%\begin{proof}
%We already know that $t_+^z$ is locally integrable for $\Re(z)>-1$, but $zt_+^{z-1}$ is not. Therefore, we need to regularize the divergent integral
%\[
%    \int_{0}^\infty zt_+^{z-1}\varphi(t)dt.
%\]
%According to the definition of weak derivative of a distribution (\textbf{Definition \ref{weakderivative}}), we have
%\begin{align*}
%    \inn{Dt_+^z}{\varphi} &=-\inn{t_+^z}{\varphi'(t)} \\
%    &=-\lim_{\eps\to 0}\int_{\eps}^\infty t_+^{z}\varphi'(t)dt.
%\end{align*}
%Let us perform integration by parts to obtain
%\[
%    -\lim_{\eps\to 0}\int_{\eps}^\infty t_+^{z}\varphi'(t)dt = -\lim_{\eps\to 0}[t^z(\varphi(t)+C)]_{\eps}^{\infty}-\lim_{\eps\to 0}\int_{\eps}^\infty zt_+^{z-1}[\varphi(t)+C]dt.
%\]
%By choosing the integration constant to be $-\varphi(0)$, the first term goes to $0$ as $\eps$ goes to $0$. Thus
%\[
%    \inn{Dt_+^z}{\varphi} = \int_{\eps}^\infty zt_+^{z-1}[\varphi(t)-\varphi(0)]dt=\inn{zt_+^{z-1}}{\varphi(t)-\varphi(0)}.
%\]
%\end{proof}

\begin{remark}\label{ImportantRemark}
\begin{enumerate}
    \item[(i)] If we consider $\varphi \in \cS(\R)$ with the additional condition $\varphi(0)=0$, a significant simplification occurs:
        \be\label{ComplexTruncatedWeakDerivativeSpecialCase}
        \inn{Dt_+^z}{\varphi(t)} = \inn{zt_+^{z-1}}{\varphi(t)}.
    \ee
    This yields the usual result for the derivative of $t_+^z$ in the distributional sense. This choice of test functions that vanish at zero may seem restrictive, but it will serve as an essential cornerstone for the forthcoming developments in later sections.
    \item[(ii)] In the broader context of distributions, the truncated power function $t_+^z$ can be analytically continued to hold for any complex power, except at non-positive integers where it exhibits poles. This is achieved by employing a method of regularization which allows us to extend the definition of the function into regions where the original integral representation may not converge.
    
%    However, for our specific case where $\Re(z)>-1$, these poles do not appear, simplifying the analysis. Despite this, it is worth noting the technique used to handle these singularities in the general case.

%    The poles of the generalized function can be effectively 'canceled out' by dividing the function by another function that has poles at the exact same points. For $t_+^z$, this normalizing function is $\Gamma(z+1)$. This normalization process, while not necessary for $\Re(z)>-1$, is a powerful tool in the analysis of distributions.

The thusly normalized function $\frac{t_+^z}{\Gamma(z+1)}$ will be denoted by $k_z(t)$. From this point on, we refer to $k_z$ as the truncated complex power function. Note that
    \[
        D k_z(t) = \frac{zt_+^{z-1}}{\Gamma(z+1)}=k_{z-1}(t)
    \]
%    The normalization by $\Gamma(z+1)$ is not only mathematically convenient but also has deep theoretical implications, connecting the study of distributions with the rich theory of the Gamma function. By continuing $k_z$ analytically to hold for any complex power except at non-positive integers, we can even make the following observation: we have already seen that for $z=0$, $k_0$ reduces to the Heaviside function. Since we have already computed its weak derivative, we can associate
%    \be\label{k-1}
%        \inn{k_{-1}}{\varphi} = \inn{Dk_{0}}{\varphi} = \inn{\delta}{\varphi},
%    \ee
%    and this can be generalized for any pole $z=-k$, where $k$ is a non-negative integer.
%
%    For a full description of the analytic continuation of $t_+^z$, which is known as \textit{Hadamard's partie finie}, and its normalization, the reader is invited to consult \cite{Gelfand1966}. This reference provides a comprehensive treatment of these topics and offers further details for those interested in a deeper dive into the subject.
\end{enumerate}
\end{remark}

\begin{figure}[H]
    \centering
    \includegraphics[scale=0.4]{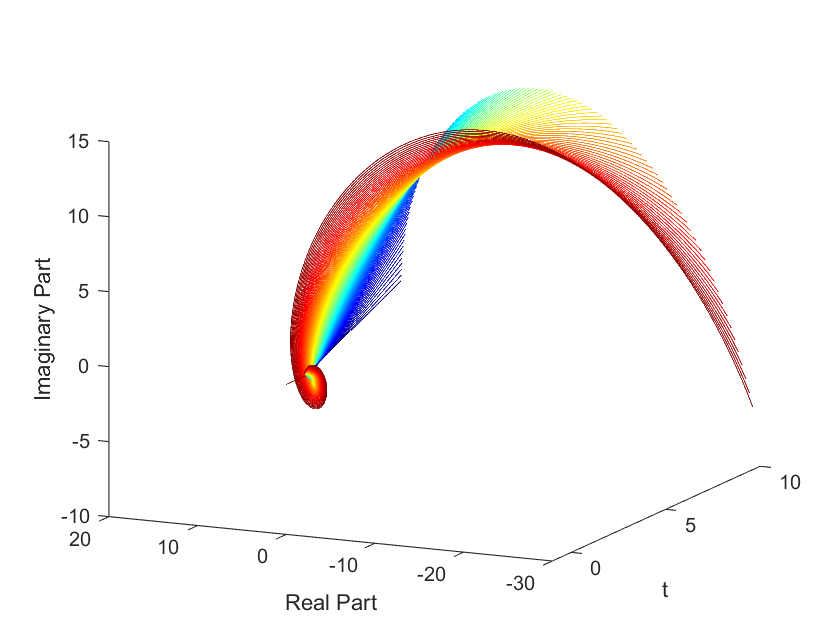}
    \caption{Normalized truncated complex power function with $z=1+i\gamma$ for $\gamma\in[0,2]$.}
    \label{NormalizedTruncatedComplexPower}
\end{figure}

%Since Matlab does not allow complex numbers as input for the Gamma function, \textit{Stirling's approximation} was used to handle the normalization factor. This may introduce some imprecision, so the graph should be interpreted with caution. For more information on Stirling's formula, readers are invited to consult \cite{Combinatorial2016}.

%Now that we have a firm understanding of the truncated complex power function and its properties, we shift our focus to the Fourier transform of this function, which plays an instrumental role in calculating the Fourier transform of complex B-splines. 

Given that complex B-splines will be constructed using truncated complex power functions, much like the cardinal B-splines in the integer case, knowing the Fourier transform of these building blocks is a crucial step in characterizing the spectral properties of the complex B-splines and later the complex box splines themselves.

%This motivates us to propose the following about the Fourier transform of the truncated complex power function:

\begin{proposition}[Fourier Transform of Truncated Complex Powers \cite{Forster2005}]\label{Fouriercomplextruncated}
For $\Re(z)>-1$, $\Re(z)\notin\N$, and $\Im (z) \neq 0$, the Fourier transform of the truncated complex power function is given by
\be\label{Fouriercomplextruncatedeq}
    \inn{\wh{k_z}}{\varphi} = \inn{\frac{1}{(i\omega)^{z+1}}}{\varphi(\omega)}.
\ee
\end{proposition}
%\begin{proof}
%Let $\varphi\in\cS(\R)$. Using the definition of the Fourier transform of a tempered distribution (\textbf{Definition \ref{FourierTempered}}), we have
%\begin{align*}
%    \inn{\wh{k_z}}{\varphi} &=\inn{k_z}{\wh{\varphi}(\omega)} \\
%    &=\frac{1}{\Gamma(z+1)}\int_{0}^{\infty}t^z\prn{\int_{\R}\varphi(\omega)e^{-i\omega t}d\omega}dt \\
%    &=\frac{1}{\Gamma(z+1)}\int_{\R}\prn{\int_{0}^{\infty}t^ze^{-i\omega t}dt}\varphi(\omega)d\omega,
%\end{align*}
%where we used the theorem by Fubini-Tonelli (\textbf{Theorem \ref{FubiniTonelli}}). We want to point out that all the conditions of the theorem are met, as the integrand $\abs{t^z e^{-i\omega t}} = t^{\Re(z)}$ is integrable because we assume $\Re(z)>-1$. Now, set $x=i\omega t$, so that $dt=\frac{1}{i\omega}dx$. Thus, we obtain
%\begin{align*}
%    \frac{1}{\Gamma(z+1)}\int_{\R}\prn{\frac{1}{(i\omega)^{z+1}}\int_{0}^{\infty}x^{z}e^{-x}dx}\varphi(\omega)d\omega &=\frac{1}{\Gamma(z+1)}\int_{\R}\frac{\Gamma(z+1)}{(i\omega)^{z+1}}\varphi(\omega)d\omega \\
%    &=\int_{\R}\frac{1}{(i\omega)^{z+1}}\varphi(\omega)d\omega \\
%    &=\inn{\frac{1}{(i\omega)^{z+1}}}{\varphi(\omega)}.
%\end{align*}
%\end{proof}
The next result whose proof can be found in \cite{FractionalEquations1999} plays an important role in the later investigations.
%Having established the Fourier transform of the truncated complex power function, we will now shift our attention to a key mathematical operation that is particularly relevant to our later discussions on complex B-splines: convolution. This understanding provides fundamental insights into the behavior of these functions and their interactions, which will be instrumental when dealing with convolutions of complex B-splines.
%
\begin{proposition}[Convolution of Truncated Complex Power Functions]\label{ConvolutionTruncated}
Let $z_1,z_2\in\C$ with $\Re(z_j)>-1$ for $j=1,2$. Then
\be\label{ConvolutionTruncatedEq}
    \inn{k_{z_1}\ast k_{z_2}}{\varphi} = \inn{k_{z_1+z_2+1}}{\varphi},
\ee
with $\varphi\in\cS(\R)$.
\end{proposition}
%\begin{proof}
%Taking the Fourier transform on the left-hand side by making use of property (iv) from \textbf{Proposition \ref{PropertiesFourier}} gives us
%    \[
%        \inn{\wh{k_{z_1}\ast k_{z_2}}}{\varphi} = \inn{\wh{k_{z_1}}\cdot \wh{k_{z_2}}}{\varphi}.
%    \]
%    Now, using the expression of the Fourier transform of $k_z$ given in \eqref{Fouriercomplextruncatedeq} yields
%    \[
%        \inn{\wh{k_{z_1}}\cdot \wh{k_{z_2}}}{\varphi(\omega)} = \inn{\frac{1}{(i\omega)^{(z_1+1)+(z_2+1)}}}{\varphi(\omega)}.
%    \]
%    We can now apply the Fourier inversion (\textbf{Definition \ref{Fourierinversion}}) to conclude
%    \[ 
%        \inn{k_{z_1}\ast k_{z_2}}{\varphi} = \inn{k_{z_1+z_2+1}}{\varphi}.
%    \]
%\end{proof}
%
%With a solid understanding of the truncated complex power function under our belt, we are now well-positioned to enter the heart of our chapter: the definition and analysis of complex B-splines. In the following subsection, we will formally introduce complex B-splines, explore their basic properties, and leverage the insights gleaned from our study of the truncated complex power function to illuminate the mathematical structure of these intriguing objects.

\subsection{Review of Complex B-Splines}

%\subsubsection{Definition and Properties}
%In this subsection, we delve into the construction of complex B-splines, drawing upon the foundational concepts discussed in the first two sections. Our approach will be to emulate the construction of cardinal B-splines, specifically by applying the complex backward difference operator to the truncated complex power function. This method allows us to extend the definition of B-splines into the complex plane, thereby broadening their applicability and potential for further mathematical exploration.
Complex B-splines were first introduced in \cite{Forster2005} and later investigated in a albeit incomplete series of papers \cite{FGMS,FM,Forster2010,MF}. The original definition was based on the Fourier transform but for our purposes we present in this section an alternate but equivalent definition via complex truncated power functions. Moreover, our definition employs a distributional setting.

\begin{definition}[Complex B-Splines]
Let $\Re(z)>-1$. The \textit{complex B-spline} $B_z$ of degree $z$ is defined as the tempered distribution given by 
%taking the $(z+1)$-th complex backward difference operator to the truncated power function
\be\label{ComplexBSplinesEq}
    \inn{B_z}{\varphi} := \inn{\nabla^{z+1}k_z}{\varphi},
\ee
with $\varphi\in\cS(\R)$.
\end{definition}

%Given the preliminary groundwork and discussions on the complex backward difference operator and truncated complex power functions, the ensuing properties of complex B-splines occur straightforwardly. The forthcoming proposition offers a precise summary of these properties.

In the next proposition, we list some properties of complex B-splines in the distributional setting. Proofs in the non-distributional setting can be found in \cite{Forster2005}; the extension to the distributional setting is based on the distributional definitions of the complex backward difference operator and complex truncated power function. The details are left to the reader.

\begin{proposition}[Properties of Complex B-Splines]
%\footnote{\cite{Unser2000}. The results are stated and proven in the fractional case, that is $\Re(z)=0$. Adequate adaptions have been made to translate the results to the complex plane.}]  
Complex B-splines enjoy the following properties:
\begin{enumerate}
    \item[(i)] The complex B-spline satisfies the convolution property
    \[
        \inn{B_{z_1}\ast B_{z_2}}{\varphi} = \inn{B_{z_1+z_2+1}}{\varphi},
    \]
    with $\Re(z_j)>-1$ for $j=1,2$.
    
    \item[(ii)] The complex B-spline satisfies the recurrence relation
    \[
        \inn{B_z}{\varphi} = \inn{\frac{t}{z}B_{z-1}+\frac{z+1-t}{z}\tau_1B_{z-1}}{\varphi},
    \]
    with $\Re(z)>0$.
    \item[(iii)] The Fourier transform of the complex B-spline $B_z$ is given by the tempered distribution
\be\label{Fouriercomplexbsplineeq}
    \inn{\wh{B_z}(\omega)}{\varphi(\omega)} = \inn{\prn{\frac{1-e^{-i\omega}}{i\omega}}^{z+1}}{\varphi(\omega)},
\ee
with $\varphi\in\cS(\R)$.
\item[(iv)] The Fourier transform $\wh{B_z}$ of a complex B-spline can be identified with an $L^2(\R)$ function if $\Re(z)>-\frac{1}{2}$. Thus, by Parseval's identity, $B_z\in L^2(\R)$ if $\Re(z)>-\frac{1}{2}$
\end{enumerate}
\end{proposition}
Note, that the spectrum of a complex B-spline $\wh{B_z}$ consists of the spectrum of a so-called fractional B-spline \cite{Unser2000}, combined with a modulation and a damping factor:
\be\label{spec}
{\widehat{B_z}(\omega)} = {\widehat{B_{\re z}}(\omega)}\, e^{i \im z \ln \Omega (\omega)} \, e^{- \im z \arg \Omega(\omega)}.
\ee
where $\Omega(\omega) := \frac{1-e^{-i\omega}}{i\omega}$.

For further properties of complex B-splines, we refer the interested reader to the albeit incomplete list of papers given at the beginning of this subsection.
\section{Box Splines}

For the purposes of notation and terminology, we present a brief introduction to box splines in this section. For more details, we refer the interested reader to, for instance, \cite{Chui1988,deBoor1993} and the references given therein.

%
%\subsection{Multivariate Splines and Direction Sets}
%
%B-splines can provide a powerful tool for representing and analyzing data in a univariate setting. However, real-world data often exhibit complex multidimensional patterns that cannot be fully captured by univariate splines. This requires a more flexible, multivariate approach. Therefore, we enter the realm of multivariate splines.
%
%Multivariate splines extend the concept of univariate splines to multiple dimensions, allowing us to capture and analyze more intricate patterns in the data. However, this extension is not straightforward. Unlike cardinal B-splines, which were defined in the knot sequence $\Z$, multivariate splines require a more complex structure to define their domain. Here, direction sets come into play.

\begin{definition}[Direction Set]
Let $n\in\N$. A direction set $M_{n+1}$ is a $d\times (n+1)$ matrix with columns $(m_0,m_1, \ldots, m_n)$, all in $\R^d\backslash 0$, such that
\[
    \Span(M_{n+1})=\R^d.
\]
\end{definition}

Direction sets are used to define box splines.
\begin{definition}[Box Spline]\label{boxspline}
Let $M = [m_0, m_1, \ldots, m_n]$ be a matrix whose columns $m_j$ form a direction set. The box spline $B_M$ associated with the matrix $M$ is defined to be the tempered distribution
\begin{align*}
    B_M : \cS(\R^d) &\rightarrow \R \\
    \varphi &\mapsto \inn{B_M}{\varphi} \coloneqq \int_{[0,1)^{n+1}} \varphi(Mt)dt.
\end{align*}
\end{definition}

The definition of the box spline $B_M$ can be understood geometrically as a mapping that measures the 'average' value of $\varphi$ over the parallelepiped defined by the direction set. This is a natural extension of the univariate case, where the B-spline measures the average value of a test function over an interval.

In \cite{Xu2009}, a generalization of box spline was presented by including a function $f$ that will serve as a weight. These so-called {F-box splines} $B_f(\cdot|M)$ are tempered distributions defined by
\begin{align}\label{FBoxSplines}
\begin{split}
    B_f(\cdot|M): \cS(\R^d) &\rightarrow \R \\
    \varphi &\mapsto \inn{B_f(\cdot|M)}{\varphi} \coloneqq \int_{\R^d}B_f(x|M)\varphi(x)dx = \int_{[0,1)^{n+1}}f(t)\varphi(Mt)dt.
\end{split}
\end{align}
%We recognize that $B_f(\cdot|M)$ is reduced to the definition of box splines by setting $f\equiv 1$.

%The introduction of the function $f$ in the definition of F-box splines opens up a new avenue of exploration. 

This generalization allows us to consider a wider class of functions as weights, thereby enriching the structure and properties of the resulting splines. One particularly interesting and important class of functions that we can consider is the class of complex-valued functions. By choosing a specific weight $f$, we can define {complex box splines}, which is the main focus of this paper.

Multivariate truncated powers \cite{deBoor1993} are defined within an affine cone
%\footnote{An affine cone defined by a direction set $X_{n+1}$ is the set of all linear combinations of the vectors $x_0,x_1,\ldots,x_n$ with non-negative coefficients.} 
with the assumption that the first non-zero component of each vector in the direction set is positive. 

\begin{definition}[Multivariate Truncated Power]
Let $M$ be an $d\times (n+1)$ matrix whose columns form a direction set. Multivariate truncated powers $T_M$ associated with $M$ are the tempered distributions given by
\begin{align}\label{MultivariateTruncated}
\begin{split}
    T_M: \cS(\R^d) &\rightarrow \R \\
    \varphi &\mapsto \inn{T_M}{\varphi} \coloneqq \int_{\R_+^{n+1}}\varphi(Mt)dt.
\end{split}
\end{align}
\end{definition}

By a weight $f$, we obtain {multivariate F-truncated powers} $T_f(\cdot|M)$:
\begin{align}\label{MultivariateFTruncated}
\begin{split}
    T_f(\cdot|M): \cS(\R^d) &\rightarrow \R \\
    \varphi &\mapsto \inn{T_f(\cdot|M)}{\varphi} \coloneqq \int_{\R^d}T_f(x|M)\varphi(x)dx = \int_{\R_+^{n+1}}f(t)\varphi(Mt)dt.
\end{split}
\end{align}

The multivariate truncated power is closely related to the box spline, as the following theorem suggests.
\begin{theorem}[Box Spline Representation \cite{deBoor1993}]
The box spline $B_M$ can also be defined as 
\be\label{boxsplinetruncated}
    \inn{B_M}{\varphi} = \inn{\prod_{j=0}^n\nabla_{m_j}T_M}{\varphi},
\ee
where $\nabla_{m_j}$ denotes the multivariate backward difference operator with respect to the vector $m_j$ and is defined as
\[
    \inn{\nabla_{m_j}T_M}{\varphi} = \inn{T_M-\tau_{m_j}T_M}{\varphi},
\]
where $\varphi\in\cS(\R^d)$.
\end{theorem}

%Equation \eqref{boxsplinetruncated} serves as a multivariate analogue to equation \eqref{dividedMarsdenFinalB}, extending the univariate relationship between the backward difference of the truncated power function and the box spline into a multivariate context, thus generalizing this remarkable result to multiple dimensions.

%\subsubsection{Properties of Box Splines}

%In this subsection, we explore some of the properties of box splines. These properties provide valuable insights into the behavior and characteristics of box splines. While we will outline these properties in this section, it is important to note that we will not be providing proofs here, as they can be found in \cite{deBoor1993}. Instead, we will reserve the proofs for the upcoming chapter, where we will tackle them in a broader and more general setting. This approach will allow us to explore these properties in a more comprehensive manner, taking into account a wider range of scenarios and conditions.

Finally, we list some properties of box splines. Proofs can be found in \cite{deBoor1993}.

\begin{proposition}[Properties of Box Splines]\label{Propertiesboxspline}
Box splines have the following properties:
\begin{enumerate}
    \item[(i)] $\Supp(B_M)=M([0,1)^{n+1})$;
    \item[(ii)] Recurrence relation:
    \[
        \inn{B_{M}}{\varphi} = \inn{\int_{0}^{1}B_{M\setminus m_n}(\cdot-tm_n)dt}{\varphi};
    \]
    \item[(iii)] If $M$ is invertible (i.e. the direction set is a basis of $\R^d$), then $B_M$ can be associated with the normalized characteristic function of the parallelepiped $M([0,1)^{n+1})$, that it
    \[
        B_M(t) = \frac{1}{|\det(M)|}\chi_{M([0,1)^{n+1})}(t);
    \]
    \item[(iv)] Fourier transform:
    \[
        \inn{\wh{B_M}}{\varphi} = \inn{\prod_{j=0}^n\frac{1-\exp(-i\omega\cdot x_j)}{i\omega \cdot x_j}}{\varphi},
    \]
    with $\varphi\in\cS(\R^d)$.
    \item[(v)] The convolution of two box splines yields again a box spline:
    \[
        \inn{B_M\ast B_N}{\varphi} = \inn{B_{M\cup N}}{\varphi},
    \]
    where $M\cup N$ is the matrix made up of the columns of $M$ and $N$.
\end{enumerate}
\end{proposition}

%\begin{figure}[H]
%    \centering
%    \includegraphics[scale=0.4]{Images/BoxSpline.png}
%    \caption{Box spline associated with the direction set $M_1=\left\{\begin{pmatrix}
%2\\
%1
%\end{pmatrix},\begin{pmatrix}
%1\\
%3
%\end{pmatrix}\right\}$.}
%    \label{BoxSpline}
%\end{figure}
%
\section{Complex Box Splines}

%In the previous chapters, we have laid a robust foundation in the rich domains of complex B-splines and box splines. Our exploration of their structure, construction, and unique properties has prepared the ground for the ambitious undertaking of the present chapter: the development of a unified theory of \textit{complex box splines}. This innovative step is inspired by the two pivotal references \cite{Xu2009} and \cite{Forster2005} and reflects my personal contribution to the field, meticulously relying on the insights gleaned from the preceding chapters.

Complex box splines are a new tool in spline theory, merging the complex degrees of B-splines with the multivariate flexibility of traditional box splines. They are a novel extension of traditional box splines and are characterized by their complex degrees, thus offering a unique perspective on the interplay between geometry and analysis, particularly in the context of multivariate functions and introducing a new multidimensional complex transform.

%The key properties of complex box splines, many of which parallel those of complex B-splines and box splines, will be our primary focus. As we venture into this remarkable realm of complex box splines, we not only expand our mathematical horizon, but also contribute an original perspective to the ongoing dialogue in this field. Our journey through this chapter promises a stimulating intellectual experience, one that reiterates the elegance of mathematics in unifying diverse concepts into a cohesive whole.

\subsection{Multivariate Truncated Complex Powers}
%\subsubsection{The Structure of Multivariate Truncated Complex Powers}

%As we embark on our exploration of complex box splines, our first stop is the concept of \textit{multivariate truncated complex powers}. In the previous chapter, we have seen how the concept of multivariate truncated powers plays a crucial role in the definition of box splines. By extending this concept to the complex domain, we can lay the groundwork for the development of complex box splines.

%The idea will be to make a suitable choice for the weight $f$ in the definition of the multivariate F-truncated power distribution given in (\ref{MultivariateFTruncated}).

We begin by introducing multivariate truncated complex power functions.

\begin{definition}[Multivariate Truncated Complex Power Function]\label{MultivariateComplexPower}
Let $M$ be an $d\times (n+1)$ matrix whose columns form a direction set. Set $\zvec \coloneqq (z_0,z_1,\ldots,z_n)\in\C^{n+1}$ where $\Re(z_j)>-1$ for all $j\in\N_n$ and let
\[
    k_\zvec(t) \coloneqq \frac{t_+^\zvec}{\Gamma(\zvec+1)} = \prod_{j=0}^n \frac{(t_j)_+^{z_j}}{\Gamma(z_j+1)}.
\]
%that is the direct product of $n+1$ univariate truncated complex power functions. 
The tempered distribution $\cT_{\zvec}(\cdot|M) \coloneqq T_{k_\zvec}(\cdot|M)$, \begin{align*}
    \cT_{\zvec}(\cdot|M): \cS(\R^d) &\rightarrow \R \\
    \varphi &\mapsto \inn{\cT_{\zvec}(\cdot|M)}{\varphi} \coloneqq \int_{\R^d}\cT_{\zvec}(x|M)\varphi(x)dx = \int_{\R_+^{n+1}}k_\zvec(t)\varphi(Mt)dt,
\end{align*}
is called a multivariate truncated complex power (function).
\end{definition}

The next result examines the rate of change of a multivariate complex truncated power distribution in the specific direction defined by an element in the direction set.

\begin{theorem}[Weak Derivative of Multivariate Truncated Complex Powers]
If $\Re(z_j)>0$ for all $j\in\N_n$, then 
\[
    \inn{D_{m_j}\cT_{\zvec}(\cdot|M)}{\varphi} = \inn{\cT_{\zvec-e_j}(\cdot|M)}{\varphi},
\]
where $e_j$ denotes one of the canonical basis vector in $\R^{n+1}$.
\end{theorem}

\begin{proof}
Let $\varphi\in\cS(\R^d)$. Then,
%Using the definition of weak derivative of distributions (\textbf{Definition \ref{weakderivative}}), we have
\begin{align*}
    \inn{D_{m_j}\mathcal{T}_{\zvec}(x|M)}{\varphi(x)} &= -\int_{\R^s}T_{k_\zvec}(x|M)D_{m_j}\varphi(x)dx \\
    &= -\int_{\R_+^{n+1}}k_\zvec(t)\frac{\p \varphi(Mt)}{\p u_j}du \\
    &= \int_{\R_+^{n+1}}\frac{\p}{\p t_j}k_\zvec(t)\varphi(Mt)du\\
    &= \int_{\R_+^{n+1}}k_{\zvec-e_j}(t)\varphi(Mt)dt \\
    &= \int_{\R^d}T_{k_{\zvec-e_j}}(x|M)\varphi(x)dx\\
    &=\inn{\cT_{\zvec-e_j}(x|M)}{\varphi(x)},
\end{align*}
where we used
\begin{align*}
    \frac{\p}{\p t_j}k_\zvec(t) &=\prn{\frac{\p}{\p t_j}k_{z_j}(t_j)}\prod_{i\neq j}^n\frac{(t_i)_+^{z_i}}{\Gamma(z_i+1)} =\prn{k_{z_j-1}(t_j)}\prod_{i\neq j}^n\frac{(t_i)_+^{z_i}}{\Gamma(z_i+1)} =k_{\zvec-e_j}(t).\qedhere
\end{align*}
\end{proof}

%We can generalize this result straightforwardly when we consider mixed partial derivatives in the weak sense. The idea is to differentiate with respect to all directions in the direction set $M_n$ simultaneously and we will denote it as $D_{M_n}$, which is not a common notation. According to the definition of the multivariate complex truncated power distribution, this translates into a mixed partial derivative.
\begin{theorem}[Mixed Weak Derivative of Multivariate Truncated Complex Powers]\label{MixedTruncated}
If $\Re(z_j)>0$ for all $j\in\N_n$, then 
\[
    \inn{D_{M}\cT_{\zvec}(\cdot|M)}{\varphi} = \inn{\cT_{\zvec-1}(\cdot|M)}{\varphi},
\]
where $\varphi\in\cS(\R^d)$. In this context, $D_M$ denotes differentiation with respect to all directions specified by the columns of the matrix $M$ simultaneously, and $\zvec-1$ denotes the vector obtained by subtracting 1 from each component of $\zvec$.
\end{theorem}
\begin{proof}
The proof follows the same lines as the previous one by noticing that
\[
    \frac{\p^{n+1}}{\p t_0\p t_1\ldots\p t_n}k_\zvec(t)=k_{\zvec-1}(t),
\]
as each factor depends on a different variable.
\end{proof}

\begin{remark}
Note that in general mixed partial derivatives may not commute for distributions. In our case, because the distribution arises from a product of truncated complex power functions, the order of differentiation does not change the result, and we can talk about ``the" mixed derivative without ambiguity.
\end{remark}

This concept of mixed partial derivative allows us to generalize the observation that we made in {Proposition \ref{ComplexTruncatedWeakDerivative}} to the multivariate complex truncated power distribution.
\begin{proposition}\label{importantequation}
If $-1<\Re(z_j)<0$, for all $j\in\N_n$, then 
\[
    \inn{D_{M_n}\cT_{\zvec}(\cdot|M)}{\varphi} = \inn{\cT_{\zvec-1}(\cdot|M)}{\varphi(\cdot)-\varphi(0)},
\]
where $\varphi\in\cS(\R^d)$.
\end{proposition}
\begin{proof}
We first observe that since $(t_j)_+^{z_j}$ is locally integrable for $\Re(z_j)>-1$, the product of $n+1$ univariate such functions is also locally integrable as long as $\Re(z_j)>-1$, for all $j\in\N_n$. Therefore, we can use the technique of the univariate case $n+1$ times.
\end{proof}

%\subsubsection{The Fourier Transform of Multivariate Truncated Complex Powers}
%The Fourier transform of the complex truncated power function has provided us with tremendous insights into the nature of complex B-splines, enabling us to prove critical properties such as the convolution property. Inspired by these revelations, we now aim to extend this understanding by exploring the Fourier transform in the multivariate case, a step that promises to unveil deeper connections and further enrich our understanding of complex box splines.

\begin{theorem}[Fourier Transform of Multivariate Truncated Complex Powers]
The Fourier transform of a multivariate truncated power is the tempered distribution given by 
\begin{equation}\label{FourierMultivariateComplex}
    \inn{\wh{\cT_{\zvec}}(\omega|M)}{\varphi(\omega)} = \inn{\prod_{j=0}^n\frac{1}{(i\omega\cdot m_j)^{z_j+1}}}{\varphi(\omega)}.
\end{equation}
\end{theorem}
\begin{proof}
Let $\varphi\in\cS(\R^d)$. Then,
\begin{align*}
    \inn{\wh{\cT_{\zvec}}(\omega|M)}{\varphi(\omega)} &= \int_{\R^d}\wh{\cT_{\zvec}}(\omega|M)\varphi(\omega)d\omega\\
    & = \int_{\R^d}\cT_{\zvec}(\omega|M)\wh{\varphi}(\omega)d\omega \\
    &= \int_{\R^{n+1}}k_{\zvec}(t)\wh{\varphi}(Mt)dt\\
    & = \int_{\R^{n+1}}k_{\zvec}(t)\prn{\int_{\R^d}\varphi(\omega)e^{-i(Mt)\cdot \omega}d\omega}dt \\
    &= \int_{\R^d}\prn{\int_{\R_+^{n+1}}k_{\zvec}(t)e^{-i(Mt)\cdot \omega}dt}\varphi(\omega)d\omega\\
    & = \int_{\R^d}\prn{\int_{\R_+^{n+1}}\prod_{j=0}^nk_{z_j}(t_j)\exp\left(-i\sum_{j=0}^n(\omega\cdot m_j)t_j}dt_j\right)\varphi(\omega)d\omega \\
    &= \int_{\R^d}\prn{\prod_{j=0}^n\int_{\R_+}k_{z_j}(t_j)e^{-it_j(\omega\cdot m_j)}dt_j}\varphi(\omega)d\omega,
\end{align*}
where we have used the Fubini-Tonelli theorem. Using the Fourier transform of the truncated complex power yields
\[
    \inn{\wh{\cT_{\zvec}}(\omega|M)}{\varphi(\omega)} = \inn{\prod_{j=0}^n\frac{1}{(i\omega\cdot m_j)^{z_j+1}}}{\varphi(\omega)},\quad \varphi\in\cS(\R^d).\qedhere
\]
\end{proof}

We now can turn our attention to the convolution of multivariate truncated complex powers, which is a natural extension of {Proposition \ref{ConvolutionTruncated}}. However, we need to be careful about what convolution means when a direction set comes into play. Indeed, a similar convolution property as in the univariate case will hold, but the resulting multivariate truncated complex power will be defined on the union of both direction sets. In order to simplify notation, we will assume that both multivariate truncated complex powers will be defined over the same direction set or the second over a subset of the first one. 

\begin{proposition}[Convolution of Multivariate Truncated Complex Powers]\label{ConvolutionMultivariateTruncated}
Let $\zvec,\overline{\zvec}\in\C^{n+1}$ with $\Re(z_j)>-1$ and $\Re(\overline{z_j})>-1$ for $j\in\N_n$. Then
\be\label{ConvolutionMultivariateTruncatedEq}
    \inn{\cT_{\zvec}(\cdot|M)\ast \cT_{\overline{\zvec}}(\cdot|M)}{\varphi} = \inn{\cT_{\zvec+\overline{\zvec}+1}(\cdot|M)}{\varphi},
\ee
with $\varphi\in\cS(\R^d)$.
\end{proposition}
\begin{proof}
Taking the Fourier transform on the left-hand side yields
% by making use of property (iv) from \textbf{Proposition \ref{PropertiesFourier}} gives us
    \[
        \inn{\cT_{\zvec}(\cdot|M)\ast \cT_{\overline{\zvec}}(\cdot|M)}{\varphi} = \inn{\wh{\cT_{\zvec}}(\omega|M)\cdot \wh{\cT_{\overline{\zvec}}}(\omega|M)}{\varphi(\omega)}.
    \]
    Now, using the expression of the Fourier transform given in \eqref{FourierMultivariateComplex} produces
    \[
        \inn{\wh{\cT_{\zvec}}(\omega|M)\cdot \wh{\cT_{\overline{\zvec}}}(\omega|M)}{\varphi(\omega)} = \inn{\prod_{j=0}^{n}\frac{1}{(i\omega\cdot m_j)^{(z_j+1)+(\overline{z_j}+1)}}}{\varphi(\omega)}.
    \]
Finally, Fourier inversion gives
    \[ 
        \inn{\cT_{\zvec}(\cdot|M)\ast \cT_{\overline{\zvec}}(\cdot|M)}{\varphi} = \inn{\cT_{\zvec+\overline{\zvec}+1}(\cdot|M)}{\varphi}.\qedhere
    \]
\end{proof}

%With a comprehensive grasp of the multivariate truncated complex power distribution, we are ready to embark on the central journey of our chapter: the definition and exploration of complex box splines. In the subsequent section, we will formally present complex box splines, investigate their fundamental characteristics, and utilize the understanding cultivated from our examination of the truncated complex power function to shed light on the intricate mathematical framework underlying these captivating constructs.

\subsection{Introducing Complex Box Splines}\label{ComplexBoxSpline}

We now can proceed to the definition of complex box splines. The idea is based on F-box splines \ref{FBoxSplines} with a suitable candidate for the weight $f$.

\begin{definition}[Complex Box Splines]
Let $M$ be an $d\times (n+1)$ matrix whose columns form a direction set. Let $\zvec \coloneqq (z_0,z_1,\ldots,z_n)\in\C^{n+1}$ such that $\Re(z_j)>-1$, for all $j\in\N_n$. Define
\[
    B_{\zvec}(t)\coloneqq \prod_{j=0}^n B_{z_j}(t_j),
\]
and
\[
    \cB_{\zvec}(\cdot|M)\coloneqq T_{B_{\zvec}}(\cdot|M).
\]
The tempered distribution $\cB_{\zvec}(\cdot|M)$ given by
\be\label{complexboxsplineEq}
\begin{aligned}
    \cB_{\zvec}(\cdot|M): \cS(\R^d) &\rightarrow \R \\
    \varphi &\mapsto \inn{\cB_{\zvec}(\cdot|M)}{\varphi} \coloneqq\int_{\R^d}\cB_{\zvec}(x|M)\varphi(x)dx =\int_{\R_+^{n+1}}B_{\zvec}(t)\varphi(Mt)dt
\end{aligned}
\ee
is called a {box spline of complex order $\zvec$}, for short, a complex box spline.
\end{definition}

Note that, unlike in the definition of F-box splines, the integral defining complex box splines is over $\R_+^{n+1}$. Indeed, this is due to the fact that complex B-splines are not compactly supported.

Now, assume for a moment that $M$ is invertible, i.e., the direction set is a basis of $\R^d$. Set $y := Mt$ in the definition of complex box spline \eqref{complexboxsplineEq}. Then, we obtain
\[
    \inn{\cB_\zvec(y|M)}{\varphi(y)} = \inn{\frac{1}{\abs{\det(M)}}\int_{M([0,1)^n)} B_\zvec(M^{-1}y)\phi(y)dy}{\varphi(y)}.
\]
Thus, $\cB_\zvec(\cdot|M)$ can be identified with the function
\be\label{identification}
    \cB_\zvec(y|M) = \frac{1}{\abs{\det(M)}} B_\zvec(M^{-1}y)\chi_{M([0,1)^n)}(y),
\ee
in the sense that the distribution acts on a test function, as the right-hand side does. Here, $\chi_S$ denotes the characteristic function of a set $S$.

\begin{figure}[H]
    \centering
    \includegraphics[scale=0.42]{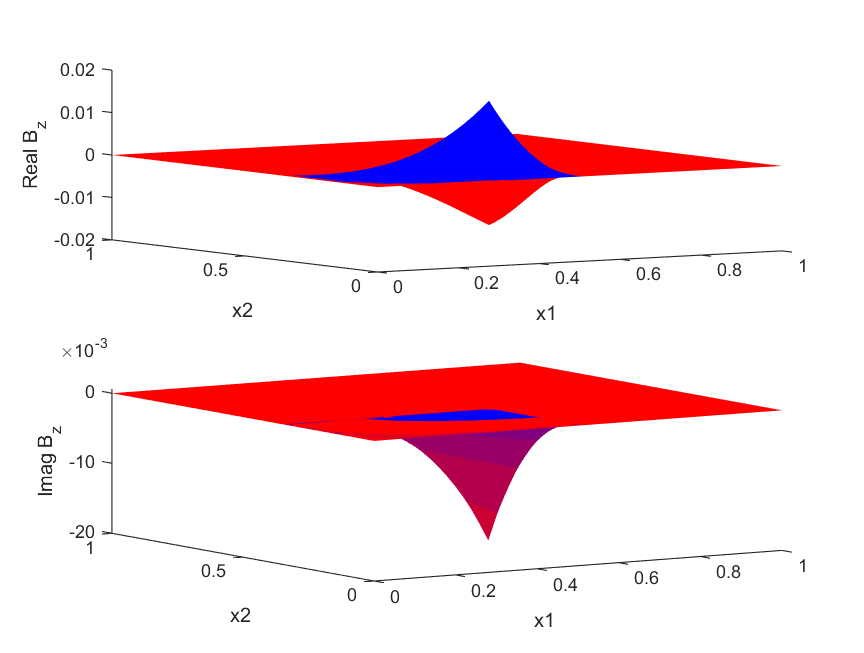}
    \caption{Real and imaginary parts of the complex box spline with $M=\begin{pmatrix}
2 & 0\\
0 & 3
\end{pmatrix}$ and $z_1=3+i$, $z_2=2+i$.}
    \label{ComplexBoxSplineFigure}
\end{figure}

%In light of the foundational work established in the previous chapters, we are now positioned to expand on these results in the context of complex box splines. We have met several properties of complex B-splines and box splines. Building directly upon these fundamental insights, we extend our investigation to complex box splines. A thorough understanding of the univariate context serves as the stepping stone that enables us to confidently delve into more complex multidimensional scenarios. 
The next theorem investigates the weak derivative of complex box splines.
\begin{theorem}[Weak Derivative of Complex Box Splines]
\begin{enumerate}
    \item[(i)] If $\Re(z_j)>0$, then
    \be\label{derivativecomplexboxspline}
        \inn{D_{m_j}\cB_{\zvec}(\cdot|M)}{\varphi} = \inn{\cB_{\zvec-e_j}(\cdot|M)-\tau_{m_j}\cB_{\zvec-e_j}(\cdot|M)}{\varphi}.
    \ee
    \item[(ii)] If $\Re(z_j)>k$, then
    \begin{equation}\label{Generalderivativecomplexboxspline}
        \inn{D^k_{m_j}\cB_{\zvec}(\cdot|M)}{\varphi} = \inn{\sum_{l=0}^{k}(-1)^l\binom{k}{l}\tau_{jm_j}\cB_{\zvec-km_j}(\cdot|M)}{\varphi}.
    \end{equation}
    \item[(iii)] If $\Re(z_j)>0$, for all $j\in\N_n$, then
    \begin{equation}\label{MixedDerivativeComplexBox}
        \inn{D_{M_n}\cB_{\zvec}(\cdot|M)}{\varphi} = \inn{\cB_{\zvec-1}(\cdot|M)-\tau_{1}\cB_{\zvec-1}(\cdot|M)}{\varphi}.
    \end{equation}
\end{enumerate}
\end{theorem}
\begin{proof}
\begin{enumerate}
    \item[(i)] Recall that the derivative of a complex B-spline is given by \cite{Forster2005}
    \[
        D B_z(t)=B_{z-1}(t)-B_{t-1}(t-1).
    \]
    Applying this result to the partial derivative of the direct product of the $n+1$ univariate complex B-splines, we get:
    \begin{align*}
        \frac{\p}{\p t_j}B_\zvec(t) &= \prn{\frac{\p}{\p t_j}B_{z_j}(t_j)}\prod_{k\neq j}B_{z_k}(t_k)\\
        & = \prn{B_{z_j-1}(t_j)-B_{z_j-1}(t_j-1)}\prod_{k\neq j}B_{z_k}(t_k) \\
        &= B_{z-e_j}(t)-\tau_jB_{z-e_j}(t).
    \end{align*}
    Given $\varphi\in\cS(\R^s)$ and using the definition of the weak derivative of a distribution, we have:
    \begin{align*}
        \inn{D_{m_j}\cB_{\zvec}(x|M)}{\varphi(x)} &= -\int_{\R^d}\cB_{\zvec}(x|M)D_{m_j}\varphi(x)dx\\
        & = -\int_{\R_+^{n+1}}B_{\zvec}(t)\frac{\p \varphi(Mt)}{\p t_j}dt \\
        &= \int_{\R_+^{n+1}}\frac{\p}{\p t_j}B_{\zvec}(t)\varphi(Mt)dt\\
        & = \int_{\R_+^{n+1}}(B_{z-e_j}(t)-\tau_jB_{z-e_j}(t))\varphi(Mt)dt \\
        &= \int_{\R^d}(\cB_{\zvec-e_j}(x|M)-\cB_{\zvec-e_j}(x-m_j|M))\varphi(x)dx \\
        &= \inn{\cB_{\zvec-e_j}(\cdot|M)-\tau_{m_j}\cB_{\zvec-e_j}(\cdot|M)}{\varphi}.
    \end{align*}

    \item[(ii)] The proof follows similarly to (i) by recalling the general recurrence formula for the derivative of the complex B-spline \cite{Forster2005}:
    \[
        D^k B_z(t) = \sum_{l=0}^{k}(-1)^j\binom{k}{l}B_{z-k}(t-l).
    \]

    \item[(iii)] The proof proceeds similar to those above, noting that:
    \[
        \frac{\p^{n+1}}{\p t_0\p t_1\ldots\p t_n}B_\zvec(t) = B_{\zvec-1}(t)-B_{\zvec-1}(t-1),
    \]
    since each factor depends on a different variable. 
\end{enumerate}
\end{proof}

\begin{theorem}[Complex Box Spline Representation]\label{Multivariatebackwardcomplex}
A complex box spline can also be defined by means of a backward difference operator:
\[
    \inn{\cB_{\zvec}(\cdot|M)}{\varphi} = \inn{\prod_{j=0}^n\nabla_{m_j}^{z_j+1}\cT_{\zvec}(\cdot|M)}{\varphi},
\]
where
\[
    \inn{\prod_{j=0}^n\nabla_{m_j}^{z_j+1}\cT_{\zvec}(\cdot|M)}{\varphi}\coloneqq\inn{\prod_{j=0}^n\sum_{k\geq 0}(-1)^k\binom{z_j+1}{k}\tau_{km_j}\cT_{\zvec}(\cdot|M)}{\varphi}.
\]
\end{theorem}
\begin{proof}
Using the definition of $\cT_{\zvec}(\cdot|M)$, we have
\begin{align*}
    \int_{\R^d}\cT_{\zvec}(x-km_0|M)\varphi(x)dx &=\int_{k}^{\infty}\int_{\R_+^{n}}k_{\zvec}(t_0-k,t_1,\ldots,t_n)\phi(Mt)dt \\
    &=\int_{\R_+^{n+1}}k_{\zvec}(t_0-k,t_1,\ldots,t_n)\varphi(Mt)dt,
\end{align*}
where the last equality follows from the fact that
\[
    k_{\zvec}(t_0-k,t_1,\ldots,t_n)=\frac{(t_0-k)_+^{z_0}}{\Gamma(z_0+1)}\cdot\ldots\cdot\frac{(t_n)_+^{z_n}}{\Gamma(z_n+1)}=0,
\]
for $t_0\leq k$. Then,
\begin{align*}
    \int_{\R^d}\nabla_{m_0}^{z_0+1}\cT_{\zvec}(x|M)\varphi(x)dx &=\int_{\R^d}\sum_{k\geq 0}(-1)^k\binom{z_0+1}{k}\tau_{km_0}\cT_{\zvec}(x|M)\varphi(x)dx \\
    &=\int_{\R_+^{n+1}}\sum_{k\geq 0}(-1)^k\binom{z_0+1}{k}k_{\zvec}(t_0-k,t_1,\ldots,t_n)\phi(Mt)dt \\
    &=\int_{\R_+^{n+1}}B_{z_0}(t_0)\frac{(t_1)_+^{z_1}}{\Gamma(z_1+1)}\cdot\ldots\cdot\frac{(t_n)_+^{z_n}}{\Gamma(z_n+1)}\phi(Mt)dt,
\end{align*}
where $B_{z_0}(t_0)$ is a B-spline of order $z_0+1$. Iterating this procedure for $z_1,\ldots,z_n$, we finally obtain
\[
    \inn{\prod_{j=0}^n\nabla_{m_j}^{z_j+1}\cT_{\zvec}(\cdot|M)}{\varphi}=\inn{\prod_{j=0}^n B_{z_j}(t_j)\cT_{\zvec}(\cdot|M)}{\varphi},
\]
where $\prod\limits_{j=0}^n B_{z_j}(t_j)$ is a tensor product of B-splines of respective orders $z_0+1,\ldots,z_n+1$.
\end{proof}

%We observe that this result is a straightforward extension of the definition of complex B-splines and that the backward difference operator behaves nicely, as it acts on test functions as a finite composition of univariate backward difference operators.

%Having established the foundational principles of complex box splines, we are now poised to go further into their mathematical intricacies. A particularly compelling aspect of these constructs lies in their recurrence relations. We will again assume that both complex box splines are defined over the same direction set or the second over a subset of the first one. 
Next, we consider the existence of recurrence formulas for complex box splines.
\begin{theorem}[Recurrence Formulas for Complex Box Splines]
Let $\varphi\in\cS(\R^d)$.
\begin{enumerate}
    \item[(i)] The complex box spline satisfies the recurrence formula
    \[
        \inn{\cB_{\zvec}(\cdot|M)}{\varphi} = \inn{\int_{0}^{\infty}B_{z_n}\tau_{um_n}\cB_{\zvec\backslash z_n}(\cdot|M\backslash m_n)du}{\varphi}.
    \]
    \item[(ii)] Let $\zvec,\overline{\zvec}\in\C^{n+1}$ with $\Re(z_j)>-1$ and $\Re(\overline{z_j})>-1$, for $j\in\N_n$. Then
    \be\label{ConvolutionComplexBoxSplinesEq}
        \inn{ \cB_{\zvec}(\cdot|M)\ast \cB_{\overline{\zvec}}(\cdot|M)}{\varphi} = \inn{\cB_{\zvec+\overline{\zvec}+1}(\cdot|M)}{\varphi}.
    \ee
\end{enumerate}
\end{theorem}
\begin{proof}
\begin{enumerate}
    \item[(i)] Starting with the right-hand side we have
\begin{align*}
    & \inn{\int_{0}^{\infty}B_{z_n}\tau_{um_n}\cB_{\zvec\backslash z_n}(\cdot|M\backslash m_n)du}{\varphi}\\ & = \int_{\R^d}\prn{\int_{0}^{\infty}B_{z_n}\cB_{\zvec\backslash z_n}(x-um_n|M\backslash m_n)du}\varphi(x)dx \\
   & = \int_{0}^{\infty}\prn{\int_{\R^d}B_{z_n}\cB_{\zvec\backslash z_n}(x-um_n|M\backslash m_n)\varphi(x)dx}du\\
    & = \int_{0}^{\infty}\prn{\int_{\R_+^{n}}B_{\zvec}(v,u)\varphi((M\backslash m_n)v+um_n)dv}du,
\end{align*}
where the second equality follows from the theorem by Fubini-Tonelli and the third by the definition of $\cB_{\zvec}(\cdot|M)$ with $v=(v_0,\ldots,v_{n-1})$. Now, setting $y=(y_0,\ldots,y_n)\coloneqq (v,u)$, we get
\[
    My=(M\backslash m_n)v+um_n.
\]
This leads to
\[
    \int_{\R_+^{n+1}}B_{\zvec}(t)\varphi(Mt)dt = \int_{\R^d}\cB_{\zvec}(x|M)\varphi(x)dx,
\]
which implies that
\[
    \int_{\R^d}\cB_{\zvec}(x|M)\phi(x)dx = \int_{\R^d}\prn{\int_{0}^{\infty}B_{z_n}\cB_{\zvec\backslash z_n}(x-um_n|M\backslash m_n)du}\varphi(x)dx,
\]
for all $\varphi\in\cS(\R^d)$.

\item[(ii)] Using the representation of complex box splines given in {Theorem \ref{Multivariatebackwardcomplex}}, we deduce
    \[
        \inn{\cB_{\zvec}(\cdot|M)\ast \cB_{\overline{\zvec}}(\cdot|M)}{\varphi} = \inn{\prod_{j=0}^n\nabla_{m_j}^{z_j+1}\cT_{\zvec}(\cdot|M)\ast \prod_{j=0}^n\nabla_{m_j}^{\overline{z_j}+1}\cT_{\overline{\zvec}}(\cdot|M)}{\varphi}.
    \]
Using the convolution property of the backward difference operator (\ref{ComplexBackwardConvolution}) (still valid since we consider a finite product), we have
    \[
        \inn{\prod_{j=0}^n\nabla_{m_j}^{z_j+1}\cT_{\zvec}(\cdot|M) \ast \prod_{j=0}^n\nabla_{m_j}^{\overline{z_j}+1}\cT_{\overline{\zvec}}(\cdot|M)}{\varphi} = \inn{\prod_{j=0}^n\nabla_{m_j}^{(z_j+1)+(\overline{z_j}+1)}\prn{\cT_{\zvec}(\cdot|M)\ast\cT_{\overline{\zvec}}(\cdot|M)}}{\varphi},
    \]
where the convolution property of the multivariate truncated complex powers (\ref{ConvolutionMultivariateTruncatedEq}) yields
    \[
        \inn{\prod_{j=0}^n\nabla_{m_j}^{(z_j+1)+(\overline{z_j}+1)}\prn{\cT_{\zvec}(\cdot|M)\ast\cT_{\overline{\zvec}}(\cdot|M)}}{\varphi} = \inn{\prod_{j=0}^n\nabla_{m_j}^{z_j+\overline{z_j}+1}\cT_{\zvec+\overline{\zvec}+1}(\cdot|M)}{\varphi},
    \]
given the linearity of the backward difference operator with the series defining it converging absolutely. Thus, 
    \[
        \inn{\cB_{\zvec}(\cdot|M)\ast \cB_{\overline{\zvec}}(\cdot|M)}{\varphi} = \inn{\cB_{\zvec+\overline{\zvec}+1}(\cdot|M)}{\varphi},
    \]
completing the proof.
\end{enumerate}
\end{proof}

\subsubsection{The Fourier Transform of Complex Box Splines}
We recall that the Fourier transform of a complex B-spline $B_z$ is given by
\[
\inn{\wh{B_z}(\omega)}{\varphi(\omega)} = \inn{\prn{\frac{1-e^{-i\omega}}{i\omega}}^{z+1}}{\varphi(\omega)}
\]

As with single-variable complex functions, we have to be aware of the branch cuts of the individual functions $w_j^{z_j}$, when considering a product of the form
\[
    \prod_{j=0}^n w_j^{z_j}.
\]
However, we can restrict the argument of $w_j$ to $[-\pi, \pi)$ for each function $w_j^{z_j}$, which ensures that the product is single-valued and, therefore, well-defined.

We can use this knowledge to formulate the following result.

\begin{proposition}[Fourier Transform of Complex Box Splines]
The Fourier transform of a complex box spline is given by the tempered distribution
\be\label{FourierComplexBoxSplineEq}
    \inn{ \wh{\cB_{\zvec}}(\omega|M)}{\varphi(\omega)} = \inn{\prod_{j=0}^n\prn{\frac{1-e^{-i\omega\cdot m_j}}{i\omega\cdot m_j}}^{z_j+1}}{\varphi(\omega)}.
\ee
\end{proposition}
\begin{proof}
Let $\varphi\in\cS(\R^d)$. Using the Fourier transform of tempered distributions, we have
\begin{align*}
    \inn{\wh{\cB_{\zvec}}(\omega|M)}{\varphi(\omega)} &= \int_{\R^d} \wh{\cB_{\zvec}}(\omega|M)\varphi(\omega) d\omega\\
    & = \int_{\R^d} \cB_{\zvec}(t|M) \wh{\varphi}(t) dt \\
    &= \int_{\R^{n+1}} B_{\zvec}(t) \wh{\varphi}(Mt) dt\\
    & = \int_{\R^{n+1}} B_{\zvec}(t) \prn{\int_{\R^d} \varphi(\omega) e^{-i(Mt)\cdot \omega} d\omega} dt \\
    &= \int_{\R^d} \prn{\int_{\R_+^{n+1}} B_{\zvec}(t) e^{-i(Mt)\cdot \omega} dt} \varphi(\omega) d\omega\\
    & = \int_{\R^d} \prn{\prod_{j=0}^n\int_{\R_+} B_{z_j}(t_j) e^{-it_j (\omega\cdot m_j)} dt_j} \varphi(\omega) d\omega,
\end{align*}
where we have used the Fubini-Tonelli theorem. Now, using the Fourier transform of the complex B-spline, we have
\[
    \int_{\R^d} \wh{\cB_{\zvec}}(\omega|M) \varphi(\omega) d\omega = \int_{\R^d} \prod_{j=0}^n \prn{\frac{1-e^{-i\omega\cdot m_j}}{i\omega\cdot m_j}}^{z_j+1} \varphi(\omega) d\omega,
\]
which holds for any $\varphi\in\cS(\R^d)$. Thus,
\[
    \inn{\wh{\cB_{\zvec}}(\omega|M)}{\varphi(\omega)} = \inn{\prod_{j=0}^n\prn{\frac{1-e^{-i\omega\cdot m_j}}{i\omega\cdot m_j}}^{z_j+1}}{\varphi(\omega)}.\qedhere
\]
\end{proof}

Now, consider the function $\Omega_j:\R^d\to\C$,
\[
    \Omega_j(\omega) := \frac{1-e^{-i\omega\cdot m_j}}{i\omega\cdot m_j}, \quad j\in\N_n.
\]
Then,
\[
    \Omega_j^{z_j+1}(\omega)=e^{(z_j+1)\log(\Omega_j(\omega))},
\]
with $\Re(z_j) > -\frac{1}{2}$, for all $j\in\N_n$. The logarithm function has a branch cut along the negative real axis. However, since $\Re(z_j) > -\frac{1}{2}$, it ensures that we stay away from the negative real axis. Hence, there are no problems with the branch cut for $\log(\Omega_j)$. Therefore, the function $\Omega_j^{z_j+1}$ is well-defined.

Given that
\[ 
    \log(\Omega_j(\omega)) = \log|\Omega_j(\omega)| + i\arg(\Omega_j(\omega)), 
\]
where $\arg(\Omega_j(\omega))$ is restricted to the range $[-\pi, \pi)$ and $\Omega_j(\omega) \neq 0$, the function $\Omega_j^{z_j+1}(\omega)$ expands to
\[ 
    \Omega_j^{z_j+1}(\omega) = \Omega_j^{\Re(z_j)+1}(\omega)\exp(i\Im(z_j)\log|\Omega_j(\omega)|)\, \exp(-i\Im(z_j)\arg(\Omega_j(\omega))). 
\]

The product of these functions is then
\[ 
    \prod_{j=0}^n \Omega_j^{z_j+1}(\omega) = \left(\prod_{j=0}^n \Omega_j^{\Re(z_j)+1}(\omega) \right) \exp\left(i \sum\limits_{j=0}^n \Im(z_j)\log|\Omega_j(\omega)|\right) \exp\left(-i \sum\limits_{j=0}^n \Im(z_j)\arg(\Omega_j(\omega))\right). 
\]

\begin{remark}
In the multivariate setting, an interpretation of the complex box spline as in \eqref{spec} becomes more intricate due to the presence of multiple frequency components due to the inner product of $\omega$ with the columns $m_j$ of the matrix $M$. The interplay between these columns and the variable $\omega$ can give rise to a variety of phase behaviors. 

To this end, suppose that $d = n+1$, so that $M$ is a square matrix. We consider the following three cases.
\begin{enumerate}
    \item[(i)] If $M$ is a diagonal matrix (i.e. $m_j$ become scalars) then the multivariate setting is reduced to a product of independent univariate settings. Each function $\Omega_j$ will simply be a scaling of the univariate function $\Omega$ by the factor $m_j$. This means that complex interplays between directions in the frequency domain are missing. Each scalar $m_j$ will simply scale the frequency.
    \begin{enumerate}[label=\arabic*.,start=1]
        \item {Directionality Removed:} The function $\Omega_j^{z_j+1}$ no longer affects a particular direction in the frequency spectrum. Instead, each $m_j$ simply scales the corresponding frequency component.
        \item {Uniform Frequency Shifting:} Given that $\omega \cdot m_j = \omega m_j$, the argument of $\Omega_j^{z_j+1}$ will be uniformly affected. This means that the frequency shifts due to the inner product will be uniform across the spectrum.
        \item {Amplification/Attenuation:} The magnitude of each scalar $m_j$ will determine the amplification or attenuation effect. Larger values of $|m_j|$ will result in more pronounced effects in the frequency domain.
        \item {Cumulative Effects:} Since all scalars are affecting the same ``universal" direction (there is only one dimension to consider), their effects will accumulate. Two large scalar values with similar $z_j$ parameters will have a compounded effect on the frequency spectrum.
        \item {Simplification:} In many ways, the scalar scenario can be viewed as multiple univariate cases combined. Each scalar $m_j$ scales the frequency, and the various $z_j$ values determine the specific transformation (shift, scaling, phase change).
    \end{enumerate}
    \item[(ii)] If $M$ is an orthogonal matrix the influence of each $m_j$ is largely independent of the others, as they point in non-overlapping directions in space. 
    \begin{enumerate}[label=\arabic*.,start=1]
        \item {Directional Independence:} Each $\Omega_j^{z_j+1}$ will primarily affect frequency components in the direction of its corresponding column $m_j$, without significant interference from other directions. This means the spectrum will show distinct shifts and enhancements corresponding to each $m_j$, and these effects can be analyzed individually.
        \item {Frequency Shifting:} The argument of $\Omega_j^{z_j+1}$ will be influenced by the inner product $\omega \cdot m_j$. As in the univariate case, if this inner product is positive or negative, it will cause frequency components to shift in the direction of $m_j$. The sign of $\Im(z_j)$ will determine the direction of this shift.
        \item {Scaling:} The magnitude of $\Omega_j^{z_j+1}$ will scale frequencies in the direction of $m_j$, and this scaling will be determined by the value of $\Re(z_j)+1$. This operates similarly to the univariate case, just restricted to the direction of $m_j$.
        \item {Phase Changes:} $\Im(z_j)$ will introduce a phase shift and a scaling factor for each function $\Omega_j^{z_j+1}$ in its own direction, independent of other directions.
        \item {Simplification:} In the overall frequency spectrum, we can expect to see distinct patterns or effects along each of the orthogonal directions $m_j$. These patterns will not blend into each other, keeping the effects of each $m_j$ largely separable.
    \end{enumerate}
    \item[(iii)] If $M$ is an invertible matrix then each column $m_j$ has a unique contribution that cannot be affected by the others.
    \begin{enumerate}[label=\arabic*.,start=1]
        \item {Directional Independence:} Each column $m_j$ points in a direction that is not affected by any combination of the other vectors. As a result, each $\Omega_j^{z_j+1}$ will have its own distinct influence on the frequency spectrum.
        \item {Frequency Spectrum Complexity:} Unlike the orthogonal case, where effects are isolated along distinct axes, the linearly independent case can produce more intricate patterns. The reason is that while the vectors do not overlap perfectly, they can still point in somewhat similar directions, leading to potential regions of compounded or interfered effects.
        \item {Frequency Shifting:} The inner product $\omega \cdot m_j$ will result in frequency shifts that are not purely along one axis or direction. Depending on the values of $\omega$ and $m_j$, and the signs of $\Im(z_j)$, these shifts can be in various directions and magnitudes.
        \item {Scaling and Phase Changes:} As before, $\Re(z_j)$ will influence the scaling in the direction of $m_j$, while $\Im(z_j)$ will impact the phase. However, due to the unique (and not necessarily orthogonal) directions of $m_j$, these scalings and phase changes can interact in more complicated ways.
        \item {Superposition:} Since the vectors are not orthogonal, the effects of different $m_j$ can superimpose. In areas of the frequency domain where the influence of multiple $m_j$ coincide, you might observe combined or counteracting effects.
        \item {Interference:} When multiple $m_j$ influence the same region of the frequency spectrum but do so with contrasting effects (due to different $z_j$ values or simply the nature of the $\Omega_j^{z_j+1}$ functions), interference patterns may emerge. This can manifest itself as unexpected amplifications or attenuations.
    \end{enumerate}
\end{enumerate}
\end{remark}

The next result tells when we can identify a complex box spline with an $L^2$-function on $\R^d$.
\begin{theorem}[$L^2(\R^d)$ Identification]
If $M$ is invertible, then the complex box spline
\be
    \cB_\zvec(y|M) = \frac{1}{\abs{\det(M)}}B_\zvec(M^{-1}y)\chi_{M([0,1)^d)}(y)
\ee
can be identified with an $L^2(\R^d)$ function, if $\Re(z_j)>-\frac{1}{2}$, for all $j\in\N_n$.
\end{theorem}
\begin{proof}
We can immediately compute
\be
    \int_{\R^d}\abs{\cB_\zvec(y|M)}^2dy=\frac{1}{\abs{\det(M)}^2}\int_{\R^d}\abs{B_\zvec(M^{-1}y)\chi_{M([0,1)^d)}(y)}^2dy.
\ee
Setting $x=M^{-1}y$, yields
\be
    \int_{\R^d}\abs{\cB_\zvec(y|M)}^2dy=\int_{\R^d}\abs{B_\zvec(x)}^2dx<\infty,
\ee
since the direct product of univariate complex B-splines is in $L^2(\R^d)$ when $\Re(z_j)>-\frac{1}{2}$, for all $j\in\N_n$.
\end{proof}

By Parseval's identity , this proposition also implies that the Fourier transform of a complex box spline can be identified with an $L^2(\R^d)$ function in the case where $M$ is invertible. Thus,
\[
    \wh{\cB_{\zvec}}(\omega|M) = \prod_{j=0}^n\prn{\frac{1-e^{-i\omega\cdot m_j}}{i\omega\cdot m_j}}^{z_j+1},
\]
in the $L^2(\R^d)$ sense.

Note that the zeros of $\wh{\cB_{\zvec}}(\omega|M)$ are given by those $\omega$ which are orthogonal to some $m_j$ for each $j$. These will be the elements that are in the intersection of the orthogonal complements of the $m_j$. However, if the matrix $M$ is diagonal, then the $m_j$ are simply the elements in the main diagonal. Since each $m_j$ is a constant, each factor in the product becomes a function of $\omega$ alone and not a scalar product with $m_j$. The function $\wh{\cB_{\zvec}}(\omega|M)$ then becomes
\[
    \wh{\mathcal{B}_{\zvec}}(\omega|M) = \prod_{j=0}^n\prn{\frac{1-e^{-i\omega m_j}}{i\omega m_j}}^{z_j+1}.
\]
The zeros of each factor are the solutions of $e^{-i\omega m_j} = 1$, i.e., $\omega \in \frac{2\pi}{m_j}\Z$. The latter set is a well-defined as we assumed that the vectors in the direction set are non-zero.

Moreover, by using l'Hopital's rule on each factor in the product, $\wh{\cB_{\zvec}}(\cdot|M)$ has a continuous extension $\wh{\cB_{\zvec}}(0|M)=1$ if $M$ is invertible.

\subsubsection{Decay Rate of the Fourier Transform}
Here, we investigate the the decay of the Fourier transform. We will see that the form of the matrix $M$ becomes important. 

First, let us suppose that $M$ is a diagonal matrix. Then, each term in the product becomes effectively decoupled from the others in terms of their dependence on $\omega$. The vectors $m_j$ align with the coordinate axes and each $\omega_j = m_j \cdot \omega$ corresponds to a single component of the vector $\omega$. The function $\wh{\cB_{\zvec}}(\cdot|M)$ then becomes a product of functions of the individual components of $\omega$ and the decay of $\wh{\cB_{\zvec}}(\omega|M)$ as $\abs{\omega} \to \infty$ is determined by the decay of the individual terms in the product.

In this case, the function $\wh{\cB_{\zvec}}(\cdot|M)$ can be interpreted as the product of univariate functions, each with its own rate of decay. The overall decay rate of $\wh{\cB_{\zvec}}(\cdot|M)$ is then determined by the term with the smallest exponent in its power-law decay. Therefore $\wh{\cB_{\zvec}}(\cdot|M)$ decays like
\be
    \frac{1}{\abs{\omega}^{\alpha+1}} \quad \text{ as } \abs{\omega}\rightarrow\infty,
\ee
where $\alpha \coloneqq \min\{\Re(z_j)\,:\, jj\in\N_n\}$. The Sobolev embedding theorem implies then that
\be
    \cB_{\zvec}(\cdot|M)\in W^{k,2}(\R^d), \quad \text{ for } k < \alpha+\tfrac{1}{2}.
\ee
Furthermore, if we assume that $k \geq \frac{d}{2}$, then the Sobolev embedding theorem allows us to conclude that $\cB_{\zvec}(\cdot|M)$ has a representative that belongs to the Hölder space $C^{l,\gamma}(\R^d)$, where $l = \floor{k - \frac{d}{2}}$ and $\gamma = k - \frac{d}{2} - l$. We can embed into the specific Hölder space $C^{\alpha,\gamma}(\R^d)$, by adjusting $k$ such that $\alpha = k - \frac{d}{2}$, while maintaining $k \geq \frac{d}{2}$.

\begin{figure}[H]
    \centering
    \includegraphics[scale=0.4]{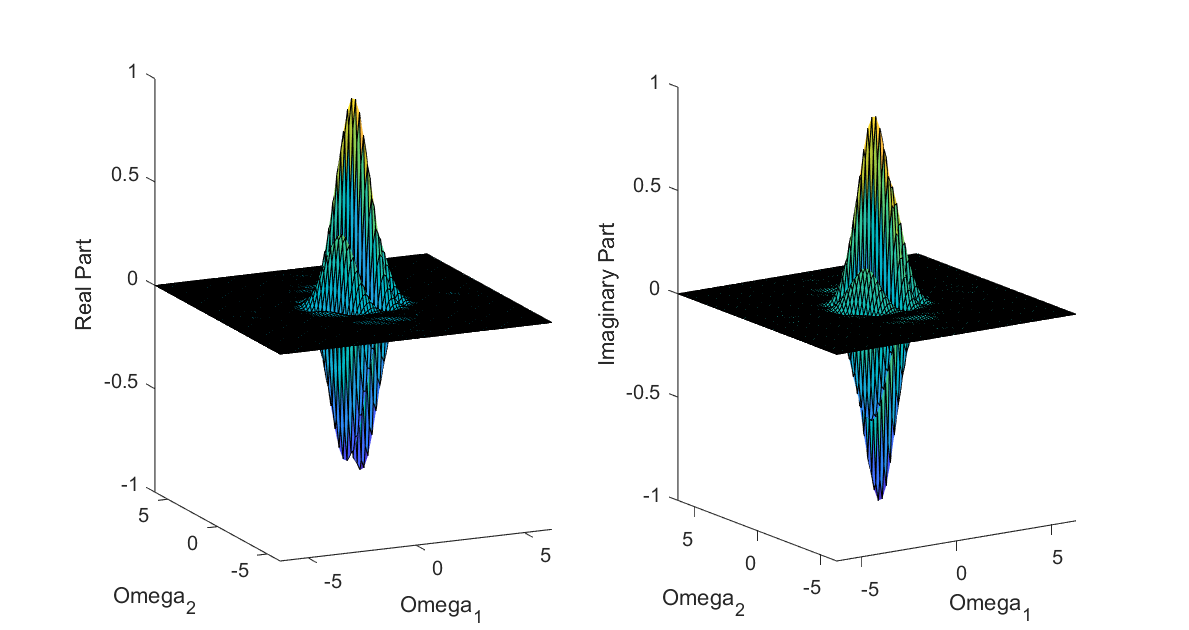}
    \caption{Real and imaginary parts of the Fourier transform of the complex box spline with $M=\begin{pmatrix}
2 & 0\\
0 & 3
\end{pmatrix}$.}
    \label{FourierComplexBoxSplinePicture}
\end{figure}

\begin{remark}
If we try to weaken the assumption on $M$ and assume, for instance, that $M$ is diagonalizable, then we transform the problem into the eigenspace of $M$. The diagonalizable matrix can be written in the form of $M = P^{-1} D P$, where $P$ is the matrix of eigenvectors and $D$ is the diagonal matrix of eigenvalues. Then, the decay rate of $\wh{\cB_{\zvec}}(\cdot|M)$ as $\abs{\omega} \to \infty$ depends not only on the minimum value of $\Re(z_j) + 1$ for $0 \leq j \leq n$ but also on the magnitudes of the eigenvalues $\lambda_j$ of $M$.

In particular, each term in the product will decay as 
\[
    \frac{1}{\abs{\lambda_j (P\omega)}},
\]
so the overall rate of decay of $\wh{\cB_{\zvec}}(\cdot|M)$ will be determined by the smallest eigenvalue in magnitude, if all $z_j$ are equal. However, if the $z_j$ values are not equal, then the situation is more complex, because the contribution to the decay of the $(z_j+1)$ power in each term needs to be taken into account. In this case, the overall decay rate could be slower or faster than $\frac{1}{\abs{\omega}^{\alpha+1}}$, depending on the values of $\Re(z_j) + 1$ and the eigenvalues $\lambda_j$ of $M$.
\end{remark}

\subsubsection{Further Properties of Complex Box Splines}
\begin{theorem}
The complex box spline generates a partition of unity:
\be
    \sum_{k\in\Z^s}\cB_\zvec(x-k|M)=1.
\ee
\end{theorem}
\begin{proof}
Using the Poisson summation formula, we have that
\be
    \sum_{k\in\Z^d}\cB_\zvec(x-k|M)=\sum_{k\in\Z^d}\widehat{\cB_\zvec}(2k\pi|M)e^{-2k\pi i x}.
\ee
We note that $\widehat{\cB^{\zvec}}(\cdot|M)$ has the continuous extension $\widehat{\cB_{\zvec}}(0|M)=1$
and
\be
    \wh{\cB_\zvec}(2k\pi|M)=\prod_{j=0}^n\prn{\frac{1-e^{-i2k\pi\cdot m_j}}{i2k\pi\cdot m_j}}^{z_j+1}=0,
\ee
if $k\in\Z^d\backslash\{0\}$. Hence,
\be
    \sum_{k\in\Z^s}\cB_\zvec(x-k|M)=1.\qedhere
\ee
\end{proof}

Complex box splines also satisfy a refinement equation.
\begin{theorem}
The complex box spline $\cB_{\zvec}(\cdot|M)$ satisfies the two-scale relation
\be
    \cB_{\zvec}\prn{\tfrac{x}{2}|M} = \sum_{k\in\Z^d}h^{\zvec}(k)\cdot\cB_{\zvec}(x-k|M),
\ee
where
\be
    h_{\zvec}(k)\coloneqq\left(\tfrac{1}{2}\right)^{\sum\limits_{j=0}^nz_j+n+1-d}\sum_{\substack{\sum_{j=0}^nt_jm_j=k \\ t_0,\ldots,t_n\in\N}}\left[\prod_{j=0}^n\binom{z_j+1}{t_j}\right].
\ee
\end{theorem}
\begin{proof}
We first observe
\begin{align*}
    \frac{\wh{\cB_{\zvec}}(\omega|M)}{\wh{\cB_{\zvec}}\prn{\frac{\omega}{2}|M}} &=\prod_{j=0}^n\prn{\frac{1-e^{-i\omega\cdot m_j}}{i\omega\cdot m_j}}^{z_j+1}\cdot \prn{\frac{i\frac{\omega}{2}\cdot m_j}{1-e^{-i\frac{\omega}{2}\cdot m_j}}}^{z_j+1} \\
    & = \prod_{j=0}^n\frac{1}{2^{z_j+1}}\prn{\frac{1-e^{-i\omega\cdot m_j}}{1-e^{-i\frac{\omega}{2}\cdot m_j}}}^{z_j+1} =\prod_{j=0}^n\frac{1}{2^{z_j+1}}\prn{1+e^{-i\frac{\omega}{2}\cdot m_j}}^{z_j+1},
\end{align*}
where the last equality follows from the relation
\[
    \frac{1-x^2}{1-x}=1+x,
\]
with $x=e^{-i\frac{\omega}{2}\cdot m_j}$ and making sure that $\omega$ and $m_j$ are not orthogonal for all $j$. Using the binomial theorem, we obtain
\[
    \prn{1+e^{-i\frac{\omega}{2}\cdot m_j}}^{z_j+1}=\sum_{t\geq 0}\binom{z_j+1}{t}e^{-ti\frac{\omega}{2}\cdot m_j}.
\]
Substituting this expression into the above equation yields
\[
    \frac{\wh{\cB_{\zvec}}(\omega|M)}{\wh{\cB_{\zvec}}\prn{\frac{\omega}{2}|M}}=\prod_{j=0}^n\frac{1}{2^{z_j+1}}\sum_{t\geq 0}\binom{z_j+1}{t}e^{-ti\frac{\omega}{2}\cdot m_j}.
\]
Distributing the product over the sum gives
\[
    \left(\tfrac{1}{2}\right)^{\sum\limits_{j=0}^nz_j+1}\sum_{t_0\geq 0}\ldots\sum_{t_n\geq 0}\left[\prod_{j=0}^n\binom{z_j+1}{t_j}\exp\left(-\sum_{j=0}^nt_j\pi i\tfrac{\omega}{2}\cdot m_j\right)\right].
\]
Now, set 
\[
    k :=\sum_{j=0}^n t_jm_j.
\]
Since the $t_j$ are non-negative integers and the $m_j$ are integer vectors, $k \in \N^d$. Then, we can rewrite the expression above as
\[
    \left(\tfrac{1}{2}\right)^{\sum\limits_{j=0}^nz_j+n+1}\sum_{k\in\Z^d}\sum_{\substack{\sum_{j=0}^n t_jm_j=k \\ t_0,\ldots,t_n\in\N}}\left[\prod_{j=0}^n\binom{z_j+1}{t_j}e^{-k\pi i \frac{\omega}{2}}\right],
\]
where the inner sum is over all combinations of $t_0,\ldots,t_n$ such that their associated $m_j$ vectors sum to $k$. This is effectively a count of how many ways we can select the $t_j$ to get each particular $k$. 

The next step is to separate out the parts of the term inside the sum that depend on $k$ and those that depend on the $t_j$. We can pull out $\exp(-k\cdot \pi i\omega)$, because it does not depend on the $t_j$. Thus, we obtain
\[
    \left(\tfrac{1}{2}\right)^{\sum\limits_{j=0}^n z_j+n+1}\sum_{k\in\Z^d}e^{-k\pi i \frac{\omega}{2}}\sum_{\substack{\sum_{j=0}^n t_jm_j=k \\ t_0,\ldots,t_n\in\N}}\left[\prod_{j=0}^n\binom{z_j+1}{t_j}\right].
\]
By defining
\[
    h_{\zvec}(k)\coloneqq\left(\tfrac{1}{2}\right)^{\sum\limits_{j=0}^nz_j+n+1-d}\sum_{\substack{\sum_{j=0}^nt_jm_j=k \\ t_0,\ldots,t_n\in\N}}\left[\prod_{j=0}^n\binom{z_j+1}{t_j}\right],
\]
we have shown the two-scale relation.
\end{proof}

\section{Fractional Differential Operators and Complex Box Splines}

%In the realm of mathematical analysis, the concept of differentiation and integration has traditionally been confined to integral orders. However, the exploration of fractional and even complex orders has opened up a new vista of possibilities, leading to the development of fractional differential and integral operators. This chapter delves into the intriguing world of these operators, which extend the conventional notions of differentiation and integration in a profound way.
%
%Fractional differential operators have a rich history that traces their roots back to the foundational work of Newton and Leibniz. Over time, these operators have evolved, with various types defined and studied. The choice of operator often depends on the specific application at hand, and different types may be more suitable for different settings.
%
%In the context of this thesis, we focus on a particular type of fractional differential and integral operator defined on specific function spaces, namely \textit{Lizorkin spaces}. These spaces allow for the identification of several fractional differential and integral operators, and they form endomorphisms that satisfy properties analogous to those of traditional integral order scenarios.

This section aims at providing a short introduction to fractional operators, elucidating their properties and potential applications, particularly in relation to complex B-splines and complex box splines. The content is based on the work presented in \cite{FractionalIntegrals1993} and \cite{Forster2010} but is presented in a novel way to ensure originality and clarity of understanding.

\subsection{Lizorkin Spaces}

Before reviewing fractional operators, it is crucial to understand the function spaces on which these operators are defined. Two of such spaces are the {Lizorkin spaces}, denoted by $\Phi(\R^n)$ and $\Psi(\R^{n})$, respectively, which are subspaces of the Schwartz space $\cS(\R^n)$. These spaces were chosen for their unique properties when applied to the Fourier transform of a function. In essence, the fractional derivative and integral in Lizorkin spaces behave similarly to their integral-order counterparts, simply multiplying by a specific factor. This facilitates the identification of several fractional differential and integral operators, and they form endomorphisms that satisfy properties analogous to those of their traditional counterparts.

\begin{definition}[Lizorkin Spaces]
The {Lizorkin spaces} $\Phi(\R^n)$ and $\Psi(\R^n)$ are defined as
\be\label{Lizorkin1}
\Phi(\R^n) \coloneqq \left\{\varphi \in \cS(\R^n) \,:\, \wh{\varphi}\in\Psi(\R^n)\right\},
\ee
where
\be\label{Lizorkin2}
    \Psi(\R^n)\coloneqq \left\{ \varphi \in \cS(\R^n) \,:\, \p^{\mu}\varphi(0) = 0, \forall \mu \in \N^n \right\}.
\ee
\end{definition}

The Lizorkin spaces play a crucial role in the theory of fractional differentiation and integration, as they allow the identification of several fractional differential and integral operators. The elements of $\Psi(\R^n)$ can be characterized as follows:

\begin{proposition}[Characterization of $\Psi(\R^n)$ \cite{Hodge2009}]
For all $\varphi\in\cS(\R^n)$, the following assertions are equivalent:
\begin{enumerate}
    \item[(i)] $\varphi\in\Psi(\R^n)$;
    \item[(ii)] $(\p^{\mu}\varphi)(\omega)\in o(\abs{\omega}^{t})$ as $\abs{\omega}\rightarrow 0$, $\forall \mu \in \N^n$ and $\forall t\in\R^+$;
    \item[(iii)] $\abs{\omega}^{-2m}\varphi\in\cS(\R^n)$, $\forall m\in\mathbb{N}$. 
\end{enumerate}
\end{proposition}

The Fourier transform has a special property: it interchanges the decay rates of a function and its derivatives.

Consider a function $\varphi \in \Phi(\R^n)$. By definition, the Fourier transform of $\varphi$, denoted by $\wh{\varphi}$, belongs to $\Psi(\R^n)$. This means that all derivatives of $\wh{\varphi}$ at the origin vanish, that is, $(\p^{\mu}\wh{\varphi})(0) = 0$ for all $\mu \in N^n$.

Now, let us consider the second condition in the previous proposition for $\Psi(\R^n)$. It states that for all $\varphi \in \Psi(\R^n)$, the derivatives of $\varphi$ at the origin are of order $o(\abs{\omega}^{t})$ as $\abs{\omega}\rightarrow 0$, for all $\mu \in N^n$ and $t\in\R^+$. This condition is satisfied by $\wh{\varphi}$, since all its derivatives at the origin vanish.

Therefore, we can see that the Fourier transform of a function in $\Phi(\R^n)$ results in a function in $\Psi(\R^n)$, which satisfies the conditions of the proposition. This shows that the Fourier transform maps $\Phi(\R^n)$ to $\Psi(\R^n)$.

We can characterize the topological dual spaces of $\Phi(\R^n)$ and $\Psi(\R^n)$ using the topological dual space $\cS'(\R^n)$ of $\cS(\R^n)$. Indeed, consider the closed
%\footnote{The set $\mathcal{P}$ is indeed closed in $\cS(\R^n)$, since the operation of taking the Fourier transform is continuous in the weak-$\ast$ topology of $\cS'(\R^n)$. Therefore, the pre-image of $\{0\}$, which is a closed set in $\R^n$, under the Fourier transform is also closed. For more information, we refer the reader to \cite{FractionalIntegrals1993} and \cite{Weak2006}} 
subset $\mathcal{P}\coloneqq \{ T \in \cS'(\R^n) \,:\, \text{supp}\left(\wh{T}\right) = \{0\}\}$ of $\cS(\R^n)$, which can be identified with the set of polynomials in $\cS'(\R^n)$.

\begin{proposition}[Characterization of Dual Spaces \cite{Hodge2009}]\label{CharacterizationDualSpaces}
The topological dual spaces of $\Phi(\R^n)$ and $\Psi(\R^n)$ can be characterized as follows:
\begin{enumerate}
    \item[(i)] $\Phi'(\R^n)=\cS(\R^n)/\mathcal{P}$;
    \item[(ii)] $\Psi'(\R^n)=\cS(\R^n)/\wh{\mathcal{P}}$.
\end{enumerate}
\end{proposition}

\subsection{Fractional Integral and Derivative Operators}

In this section, our attention is directed towards the space $\R^{n+1}$, where we introduce the notions of fractional integrals and derivatives. 
%These mathematical operators, which serve as generalizations of classical integration and differentiation, will be defined in terms of degrees to stay consistent with the thesis. 
These definitions in this subsection are precisely tailored to the context of Lizorkin spaces $\Phi(\R^{n+1})$ and $\Psi(\R^{n+1})$. Furthermore, it should be noted that from this point on, we shall presume that the support of functions residing within the space $\Psi(\R^{n+1})$ is confined to the region $[0,\infty)^{n+1}$.

\begin{definition}[Fractional Integral]
Let $\zvec \in \C^{n+1}$ be a multi-index with $\text{Re}(z_j)> -1$, for all $j\in\N_n$. The fractional integral $\cD^{-z} : \Phi(\R^{n+1}) \rightarrow \Phi(\R^{n+1})$ is defined by
\[
(\mathcal{D}^{-\zvec}\varphi)(x) := \frac{1}{\Gamma (\zvec+1)} \int_{\R^{n+1}} t_+^{\zvec} \varphi(x+t)dt.
\]
\end{definition}

\begin{definition}[Fractional Derivative]
Let $n \in \N$. For all $j\in\N_n$, let $m_j := \ceil{\text{Re}(z_j)+1}$. Set $\nu_j = m_j-z_j$ and $m = \sum\limits_{j=0}^{n} m_j$. The fractional derivative $\cD^z : \Phi(\R^{n+1}) \rightarrow \Phi(\R^{n+1})$ is given by
\[
\begin{split}
    (\mathcal{D}^\zvec\varphi)(x) &:= \frac{1}{\Gamma (\nu+1)} \frac{\p^m}{\p_{t_0}^{m_{0}}\ldots\p_{t_n}^{m_{n}}} \int_{\R^{n+1}} t^{\nu}_+ \varphi(x+t)dt \\
    &= \frac{1}{\Gamma (\nu+1)} \int_{\R^{n+1}} t^{\nu}_+ (D^m\varphi)(x+t)dt \\
    &= D^m\cD^{-\nu}\varphi.
\end{split}
\]
\end{definition}

Recalling the kernel function 
\[
    k_\zvec(t)=\frac{t_+^\zvec}{\Gamma(\zvec+1)} = \prod_{j=0}^n \frac{(t_j)_+^{z_j}}{\Gamma(z_j+1)},
\]
that was used to define multivariate truncated complex powers in {Definition \ref{MultivariateComplexPower}}, we can rewrite the definitions of the fractional derivative and integral in the following form:
\be \label{FractionalDefinitionForm}
    \mathcal{D}^\zvec\varphi = (D^{m}\varphi) \ast k_{(m-1)-(\zvec+1)}=D^{m}(\varphi\ast k_{(m-1)-(\zvec+1)}), \quad \text{where } m = \ceil{\text{Re}(\zvec)+1},
\ee
and
\[
    \mathcal{D}^{-\zvec}\varphi=\varphi\ast k_{\zvec+1}.
\]

These definitions yield a significant consequence, one that plays a vital role in our subsequent analysis. The proof of the statement can be found in  \cite{FractionalIntegrals1993}.

\begin{proposition}[Invariance of Lizorkin Spaces]
The fractional derivative and integral leave the Lizorkin space $\Phi(\R^{n+1})$ invariant.
\end{proposition}

\begin{remark}
We observe that in general
\[
    (D^{m}\varphi) \ast k_{(m-1)-(\zvec+1)} \neq D^{m}(\varphi\ast k_{(m-1)-(\zvec+1)}),
\]
when we consider function spaces other than the Lizorkin spaces. The right-hand side of (\ref{FractionalDefinitionForm}) is known as the \textit{Riemann-Liouville fractional derivative}, which is traditionally defined as a fractional integral of the derivative of a function, and the left-hand side is known as the \textit{Caputo fractional derivative}, which is another type of fractional derivative that has certain advantages over the Riemann-Liouville derivative, particularly in the physical sciences, because it allows for initial conditions to be stated in terms of function values (like in classical calculus). For more details, the reader may want to consult  \cite{FractionalEquations1999} or \cite{FractionalIntegrals1993}.
\end{remark}

Next, we summarize some properties of the fractional derivatives and integrals introduced above.

\begin{theorem}[Properties of Fractional Derivatives and Integrals \cite{FractionalIntegrals1993}]
The fractional derivatives and integrals have the following properties:
\begin{enumerate}
    \item[(i)] Semi-group property: For $\zvec,\overline{\zvec} \in \C^{n+1}$ with $\Re(z_j)> -1$ and $\Re(\overline{z_j})> -1$, for all $j\in\N_n$, we have
    \[
        \cD^{\pm(\zvec+\overline{\zvec})}\varphi = \left(\cD^{\pm\zvec}\varphi\right)\prn{\cD^{\pm\overline{\zvec}}\varphi} = \prn{\cD^{\pm\overline{\zvec}}\varphi\right)\left(\cD^{\pm\zvec}\varphi},
    \]
    where $\varphi\in\Phi(\R^{n+1})$.
    \item[(ii)] For $\zvec \in \C^{n+1}$ with $\Re(z_j)> -1$, for all $j\in\N_n$, we have
    \[
        \mathcal{D}^{\zvec}\cD^{-\zvec}\varphi=\cD^{-\zvec}\cD^{\zvec}\varphi = \varphi,
    \]
    where $\varphi\in\Phi(\R^{n+1})$.
\end{enumerate}
\end{theorem}

We observe that for $T_1,T_2\in\Psi'(\R^{n+1})$, convolution exists and is defined in the same way as for ordinary functions. The pair $(\Psi'(\R^{n+1}),\ast)$ is a convolution algebra with the Dirac delta distribution $\delta$ as its unit element.

We end this subsection by defining fractional derivatives and integrals on the dual space $\Psi'(\R^{n+1})$.

\begin{definition}[Fractional Derivative and Integral on $\Psi'(\R^{n+1})$ {\cite{FractionalIntegrals1993}}]
Let $\zvec \in \C^{n+1}$ with $\Re(z_j)> -1$ for all $j\in\N_n$ and let $T\in\Psi'(\R^{n+1})$. The fractional derivative operator $\cD^\zvec$ on $\Psi'(\R^{n+1})$ is defined by
\be\label{FractionalDerivativeDual}
    \inn{\cD^\zvec T}{\varphi} \coloneqq \inn{(D^{\ceil{\Re(\zvec)+1}}T)\ast k_{\ceil{\text{Re}(\zvec)+1}-1-(\zvec+1)}}{\varphi},
\ee
and the fractional integral operator $\cD^{-\zvec}$ by
\be\label{FractionalIntegralDual}
    \inn{\cD^{-\zvec}T}{\varphi} \coloneqq \inn{T\ast k_{\zvec+1}}{\varphi},
\ee
where $\varphi\in \Psi(\R^{n+1})$.
\end{definition}
\subsubsection{Complex B-Splines}
We can now proceed to the main goal of this section. We define complex B-splines and complex box splines as distributional solutions of a certain fractional differential equation and extend the ideas and concepts presented in, for instance, \cite{Forster2010}.

The following result was shown in \cite{Forster2010}.
\begin{proposition}
Let $z \in \C$ with $\Re(z) \geq -1$ and let $\{a_k \,:\, k \in \N\} \in \ell^{\infty}(\R)$. The complex B-spline
\[
    \inn{B_z}{\varphi} = \inn{\nabla^{z+1}k_z}{\varphi},
\]
with $\varphi \in \Psi(\R)$, is a distributional solution of the equation
\begin{equation}\label{Splinecomplexdegree}
    \inn{\mathcal{D}^z T}{\varphi} = \sum_{k \geq 0} a_k \inn{\tau_k \delta}{\varphi}.
\end{equation}
\end{proposition}

%In order to show that splines of complex degree $z$ exist, we show that complex B-splines provide a solution to (\ref{Splinecomplexdegree}).
%\begin{proposition}
%The complex B-spline
%\[
%    \inn{B_z}{\varphi} = \inn{\nabla^{z+1}k_z}{\varphi},
%\]
%with $\varphi \in \Psi(\R)$, is a distributional solution of the equation (\ref{Splinecomplexdegree}).
%\end{proposition}
%\begin{proof}
%Let $z \in \C$ with $\Re(z) \geq -1$. By the linearity of $\cD^z$ and our above computations, we obtain
%\begin{align*}
%    \inn{\cD^z B_z}{\varphi} &= \inn{\cD^z \nabla^{z+1}k_z}{\varphi} \\
%    &= \sum_{k \geq 0} (-1)^k \binom{z+1}{k} \inn{\cD^z \tau_k k_z}{\varphi} \\
%    &= \sum_{k \geq 0} (-1)^k \binom{z+1}{k} \inn{\tau_k \delta}{\varphi},
%\end{align*}
%as the series defining the complex backward difference operator converges absolutely.
%\end{proof}

\subsubsection{Complex Box Splines}

Consider $\zvec\in\C^{n+1}$ with $\Re(z_j)\geq -1$ for $j\in\N_n$ and recall the relationship between the distributional derivative of $\cT_{\zvec}(\cdot|M)$ and a supplementary condition within the range $-1 < \Re(z_j) < 0$ for all $j\in\N_n$:
\[
    \inn{D_{M_n}\cT_{\zvec}(\cdot|M)}{\varphi} = \inn{\cT_{\zvec-1}(\cdot|M)}{\varphi(\cdot)-\varphi(0)},
\]
where the additional condition $\varphi(0) = 0$ simplifies this to
\[
    \inn{D_{M_n}\mathcal{T}_{\zvec}(\cdot|M)}{\varphi} = \inn{\cT_{\zvec-1}(\cdot|M)}{\varphi(\cdot)}.
\]

This expression is consistent with the conventional result for the derivative of $\cT_{\zvec}(\cdot|M)$ in the distributional framework, which we proved in {Theorem \ref{MixedTruncated}}. We recall that by definition of multivariate truncated complex powers, this derivative was equivalent to a mixed partial derivative of $k_{\zvec}$ in the weak sense. 

Our goal is to compute $\cD^z\left[\tau_{km_j}\mathcal{T}_{\zvec}(\cdot|M)\right]$ for $\Re(z_j) \geq -1$ for all $j\in\N_n$. This is equivalent to computing $\cD^z(\tau_{k}k_{\zvec})$, where we can apply (\ref{FractionalDerivativeDual}). 
%The idea will be to mimic the univariate case and use the representation of complex box splines in terms of backward difference operators (\textbf{Theorem \ref{Multivariatebackwardcomplex}}). From our earlier insights, we consistently find that for $m = \ceil{\Re(\zvec) + 1}$, 
We recall that
\begin{equation*}
    D^{m}(\tau_{k}k_{\zvec}) = \tau_kk_{\zvec-m}.
\end{equation*}

Including the special case where $-1 < \Re(z) < 0$, we can further write:

\begin{align*} 
    \inn{\cD^z\left[\tau_{km_j}\mathcal{T}_{\zvec}(\cdot|M)\right]}{\varphi} &= \inn{\cD^z(\tau_{k}k_{\zvec})}{\varphi(M\cdot)}\\
    & = \inn{(D^{m}\tau_kk_{\zvec})\ast k_{(m-1)-(\zvec+1)}}{\varphi(M\cdot)} \\
    &= \inn{\tau_kk_{\zvec-m}\ast k_{(m-1)-(\zvec+1)}}{\varphi(M\cdot)}\\
    & = \inn{\tau_kk_{\zvec-m+(m-1)-(\zvec+1)+1}}{\varphi(M\cdot)} \\
    &= \inn{\tau_kk_{-1}}{\varphi(M\cdot)}\\
    & = \inn{\tau_k\delta}{\varphi(M\cdot)}\\
    & = \inn{\tau_{km_j}\delta}{\varphi},
\end{align*}
where we have used the convolution property of the multivariate truncated complex powers, that we showed in {Proposition \ref{ConvolutionMultivariateTruncated}}.

Now, similarly to the univariate case we can define multivariate splines via certain differential operators and introduce splines of complex degree $\zvec$ with $\Re(z_j) \geq -1$ for $j\in\N_n$ as follows.

\begin{definition}
Let $\zvec\in\C^{n+1}$ with $\Re(z_j)\geq -1$ for $j=0,\ldots,n$, and let $\{\boldsymbol{a}_k\,:\, k\in\N\}\in\ell^{\infty}(\R^{n+1})$ be a bounded sequence of vectors $\boldsymbol{a}_k = (a_{k,0}, \ldots, a_{k,n})$. A distributional solution of the equation
\begin{equation}\label{SplineComplexDegreeMultivariate}
    \inn{\cD^{\zvec}T}{\varphi} = \sum_{k\geq 0}\prn{\prod_{j=0}^{n} a_{k,j}}\inn{\tau_k\delta}{\varphi},
\end{equation}
is called a box spline of complex degree $\zvec$.
\end{definition}

In order to show that box splines of complex degree $z$ exist, we show that complex box splines provides a solution to (\ref{SplineComplexDegreeMultivariate}).

\begin{proposition}
The complex box spline
\[
    \inn{\cB_{\zvec}(\cdot|M)}{\varphi} = \inn{\prod_{j=0}^n\nabla_{m_j}^{z_j+1}\cT_{\zvec}(\cdot|M)}{\varphi},
\]
with $\varphi\in\Psi(\R^d)$, is a distributional solution of the equation (\ref{SplineComplexDegreeMultivariate}).
\end{proposition}
\begin{proof}
Let $\zvec\in\C^{n+1}$ with $\Re(z_j)\geq -1$ for $j\in\N_n$. By the linearity of $\cD^z$ and our above computations, we obtain
\begin{align*}
    \inn{\cD^z\mathcal{B}_{\zvec}(\cdot|M)}{\varphi} &=\inn{\cD^z\nabla^{z+1}k_z}{\varphi}\\
    & =\inn{\prod_{j=0}^n\sum_{k\geq 0}(-1)^k\binom{z_j+1}{k}\cD^z\left[\tau_{km_j}\cT_{\zvec}(\cdot|M)\right]}{\varphi} \\
    &=\prod_{j=0}^n\sum_{k\geq 0}(-1)^k\binom{z_j+1}{k}\inn{\cD^z\left[\tau_{km_j}\cT_{\zvec}(\cdot|M)\right]}{\varphi}\\
    & =\prod_{j=0}^n\sum_{k\geq 0}(-1)^k\binom{z_j+1}{k}\inn{\cD^z\left[\tau_{k}k_{\zvec}\right]}{\varphi(M\cdot)} \\
    &=\prod_{j=0}^n\sum_{k\geq 0}(-1)^k\binom{z_j+1}{k}\inn{\tau_k\delta}{\varphi(M\cdot)}\\
    & =\prod_{j=0}^n\sum_{k\geq 0}(-1)^k\binom{z_j+1}{k}\inn{\tau_{km_j}\delta}{\varphi} \\
\end{align*}
as the series defining the complex backward difference operator converges absolutely.
\end{proof}

\bibliographystyle{plain}
\bibliography{references}
%\printbibliography

\end{document}